\documentclass[12pt]{amsart}

\numberwithin{equation}{section}
\usepackage{amsfonts}
\usepackage{amssymb}
\usepackage{amsmath}

\begin{document}

\title {A-stabilization and the ranges of complex polynomials on the unit disk}
\author{ Alexey Solyanik }
\date{17 January  2017}
\dedicatory{To  Dmitriy Dmitrishin and Alexander Stokolos}
\address{Caribbean Sea, Grenada}
\email{transbunker@gmail.com}
\newtheorem{remark}{Remark}
\newtheorem{theorem}{Theorem}
\newtheorem{lemma}{Lemma}
\newtheorem{proposition}{Proposition}
\newtheorem{corollary}{Corollary}
\newtheorem{example}{Example}
\newtheorem*{theorem A}{ Theorem A}
\newtheorem*{theorem 4}{ Theorem 4}
\newtheorem*{caratheodory}{Caratheodory's Lemma}
\newtheorem*{rogosinskii}{Rogosinski's Lemma}
\newtheorem*{duality}{The $p/q$ duality}
\maketitle
\vspace{1 cm}
\begin{abstract}Problems of stabilization of the unstable cycle of one-dimensional complex dynamical system are briefly discussed. 
These questions reduced to the problem of description of the ranges of polynomials $q(z) = q_1z + q_2z^2 +\dots + q_nz^n$  defined in the unit disk and normalized by the conditions $q(1) = 1 $ and this is the main subject of the present paper.
\end{abstract}

\section{Introduction}
In the recent papers \cite{dm1}, \cite{dm2}, \cite{dm3} authors studied the problem of local stabilization of unstable cycle of length $T$ of the given real one-dimensional dynamical system. A typical example is the family of \textit{logistic maps} $g_\lambda (x)=\lambda x(1-x)$ defined on the interval $[0,1]$ for $\lambda \in [1,4]$  and demonstrate a \textit{chaotic dynamics} for a large set of $\lambda $ close to $4$.

In this work we  study this stabilization process  but for the \textit{complex} one-dimensional dynamical system 
\begin{equation}
\label{intr_def}
f : \hat{\mathbb{C}}  \to \hat{\mathbb{C}}  
\end{equation}  
where $\hat{\mathbb{C}}  = \mathbb{C} \cup \lbrace \infty \rbrace$ is the Riemann sphere with spherical metric.  This is of course a \textit{very mild} generalization, and generalization actually was not a target of the present work. Most of the results are new even in the real case, since all of the dynamics of $f$ restricted to $\mathbb{R} $ are contained in the dynamics on $\mathbb{C}$. This notes mainly devoted to the local theory  and the global behaviour will be the subject of the forthcoming paper.

One remark is in order. Since all described theory is mostly devoted to the spectral analysis of the given linear operator, we can apply these results also to the problems of stabilization of unstable cycles for smooth dynamical systems defined on the given  finite dimensional smooth manifold. 
\subsection{Complex dynamical systems}  

A basic question in the theory of dynamical systems is to study the asymptotic behaviour of orbits 
$$
\text{orb}(f,z_0)=(z_0, f(z_0), f^2(z_0), ...),
$$
where $f^n(z)=f(f( ... f(z)))$ and initial point $z_0 \in \hat{\mathbb{C}}$.

A point $s$, such that $f(s)=s$ is called a \textit{stable} (or fixed) point of the map $f$. The local behaviour of the dynamical system at the stable point governed by the eigenvalues of the linear part.  We call a fixed  point $s$ \textit{attracting } if the \textit{multiplier} $\mu  =f'(s)$ satisfied $\vert \mu \vert<1$. Then all orbits of $f$ convergence uniformly  to $s$ on the some neighbourhood of $s$ and point $s$ is also called \textit{asymptotically stable}.

A point $s$ is called \textit{periodic }if $f^T(s)=s$ for some $T$. The minimal  $T$ is its period and the orbit $\mathcal{O}=(s, f(s), f^2(s), ... f^{T-1}(s))$ is called a \textit{cycle}. We say that the cycle $\mathcal{O}$  is attracting when the \textit{multiplier of the cycle} $\mu ( \mathcal{O}) =(f^T)'(s)$ satisfied $\vert \mu  \vert < 1$.

For precise definitions and background theory of complex dynamics we refer to the books \cite{beardon}, \cite{carleson} or \cite{milnor} but in this work we will use only elementary facts from the theory. 

As a simplest (but typical) example of the complex (or holomorphic)  dynamical system (\ref{intr_def}) one can imagine iteration of $f(z)=z^2$.
 
It is clear that $s=1$ is a \textit{repelling} fixed point for this system, which means that $f(s)=s$ and, since $f'(1)=2$, all points close enough to $s$ repel by the map $f$ away from the point $s$.   Our goal is to  \textit{stabilize} this system (at least locally) at the point $s$ in such a way that all  points  close enough to $s$ became \textit{attracted} to the point $s$.  

Before we involve to the stabilization problems, we shall shortly describe the global dynamics of $f(z)$. Since $ f^n(z) = z^{2^n}$,  initial points $z_0$ with $\vert z_0  \vert< 1$ tends to $0$ and initial points  $z_0$ with $\vert z_0  \vert > 1$ tends to $\infty$. This means that this system has two attracting fixed points, namely $0$ and $\infty$ and one repelling fixed point $1$. This is a reflection of general fact, that any rational function of degree $d$ has precisely $d+1$ fixed points on $\hat{\mathbb{C} }$ counting like zeros of equation $f(z)-z=0$ (see \cite{beardon}, p. 40).

What about points on the circle $\partial \Delta= \lbrace z : \vert z \vert =1 \rbrace $, we have to say that their dynamics is quite complicated and we still do not know answers to a very simple questions.

For instance, it seems unknown any accumulated point of the sequence $\sin 2^n=  \Im (f^n(e^i))$.

Meanwhile, of course, all initial points with $\vert z_0  \vert = 1$ not leave the unit circle and even more -- the circle is both forward and backward invariant under $f$ (that is each point on $\partial \Delta$ has its entire history and future lying on $\partial \Delta$). Loosely speaking,
a unit circle looks like an unstable 'orbit' (actually invariant set) for this dynamical system. Unstable, since if some iteration $f^n(\zeta)$ fall from the unit circle inside or outside of the circle (for instance by the reason of approximation) it \textit{never} return back.

But the main difference with periodic or quasi-periodic motion is that  the dynamics of $f(z) = z^2$ on the unit circle is chaotic. The word \textit{chaotic} has many meanings, which we briefly discus here.

First of all this means that the dynamics is \textit{sensitive to the initial conditions}, i. e. two close to each other initial points $z_0$ and $w_0$ after some number of iterations produce a later changes that grows exponentially with  $ n$ (exponentially means (\ref{intr_exp})).  

It  implies that long term predictions of the chaotic system is almost impossible despite the deterministic nature of the equation.
For \textit{generic} initial points their trajectories, which are at first very close, later diverge more and more rapidly until they no longer have anything to do with each other.

Next, the dynamics is \textit{topologically transitive}, which means that some point (and in fact any generic point ) has everywhere dense orbit in $\partial \Delta$ (see \cite{milnor}, p. 51). 

Since the equation $z^{2^T} = z$ for every $T$ has $2^T-1$ roots, which are equally distributed on the circle, any such points is periodic with period at least $T$ (but the least period could be less of course). Multipliers of these  cycles $\mu =(f^T)'(z)= (z^{2^T})'$ satisfies $\vert \mu \vert > 1$,
and hence they are all \textit{repelling cycles}. We shall claim also, that for each natural number $T$ there are periodic points with exact period $T$ ($\eta_T=\exp( \frac{2 \pi i}{2^T-1})$) and hence this dynamical system has periodic orbits of any order.

The most common definition of \textit{chaotic behaviour} is to say  that $f$ has \textit{sensitively dependence on the initial conditions}, $f$ is \textit{topologically transitive}  and \textit{periodic orbits are dense}(\cite{devaney}, p. 269), which are all satisfied by the system $f(z)=z^2$ on the invariant set $\partial \Delta$ . In fact, for compact infinite metric spaces and continues transformations, the last two conditions implies the first one.

We conclude our brief description of chaotic behaviour pointed out that the map $f(z)=z^2$  \textit{on the unit circle} clearly demonstrates two key features of chaos --- \textit{stretching and folding}: the map stretched the unit circle in $2$ times and $2$ times folded it. Loosely speaking, stretching mechanism is responsible for sensitivity to initial conditions while the folding mechanism is responsible for topological mixing (topological transitivity and density of periodic orbits).

Thus, in this case the so called  \textit{Fatou set} is $\hat{\mathbb{C}}\setminus \partial \Delta$, with the regular dynamics of $f$ and the \textit{Julia set }is the unit circle $\partial \Delta$,  where dynamics is chaotic (and ergodic with respect to the one-dimensional Lebesgue measure on the circle).

Actually this is a general case at least for all rational functions $f$ with degree not less than $2$ -- we have partition of the Riemann sphere into two disjoint invariant sets, on one of which, so called \textit{Fatou set}, $f$ is well-behaved as dynamical system, on the other of which, the \textit{Julia set}, $f$ has chaotic behaviour, i.e. sensitive to the initial conditions, topologically transitive and repelling periodic points are dense. Excluding some exceptional cases, Julia set is a fractal set with Hausdorff dimension bigger than $1$ and corresponding dynamics is ergodic on the Julia set with respect to the corresponding Hausdorff measure.

We note that \textit{any holomorphic map} from $\hat{\mathbb{C}} $ to  $\hat{\mathbb{C}} $ is the \textit{rational function} $R$ and any rational function demonstrate interplay between \textit{expanding} and \textit{contracting} features, which is the main reason of complicated dynamics.

Indeed,  any rational function of degree $d$ folding   $d$ times the Riemann sphere and hence, \textit{in average}, expanding space in $d$ times. On the other hand it has $2d-2$ critical points, counting multiplicity, where $R'(z )=0$ (or $R$ has a pole of order two and higher) and hence $R$ highly contracting local neighbourhoods of those points (see e. g. \cite{mcmullen})

However we have keep in mind, that for some rational functions $f$ the Julia set $J(f)$ equal   $\hat{\mathbb{C}}$, the fact first discovered by Ernst Schroder in 1871 and then rediscovered in greater generality by Samuel Late in 1918 (\cite{milnor}, p. 70). The simplest example is $f(z)= 1-2/z^2$ (see \cite{beardon}  p.76).

If $J(f)\neq \hat{\mathbb{C}}$ then it is a closed  set without isolated points and  with empty interior and Fatou set is open set with at most countably many open components $\Omega$, such that $f(\Omega)$ is also open component of Fatou set. Some times a given component $\Omega$ is stable, i. e. $f(\Omega)=\Omega$ (like in our example $f(z)=z^2$) and some times not (like for $f(z)=z^2-1$). 

In the latter case the corresponding Julia set is the typical fractal set called the Basilica di San Marco after John Hubbard or shortly \textit{basilica }  and Fatou set has infinitely many open components $\Omega$ (\cite{beardon}, p. 13). The component $\Omega_0$ containing point $0$ is periodic with period $2$, i. e. $f^2(\Omega_0)=\Omega_0$, as well as component $\Omega_{-1}$ containing point $-1$. This is because points $0$ and $-1$ form a $2$-cycle. Other components $\Omega$ are eventually periodic, i. e. component became periodic with period two after some number of iterations. The component $\Omega_{\infty}$, contained $\infty$ is stable, i. e. $f(\Omega_{\infty})=  \Omega_{\infty}$.

According to the famous theorem of Dennis Sullivan, solved  the 60 year old Fatou - Julia problem on wandering
domains (\cite{beardon}, p. 176), this is a general case --- there is no wandering components in Fatou set and  every component $\Omega$ is eventually periodic (or stable). 

Points of the Fatou set is exactly points which are stable in the Lyapunov sense, i. e. for every  point $z_0$ from the Fatou set every  point $w_0$ close enough to the point $z_0$ remains uniformly close for all iterations. Hence there is no repelling cycles in the Fatou set and all repelling cycles are in the Julia set (see \cite{beardon},  p. 109).

The number of components of the Fatou set could be 0, 1, 2 or $\infty$ (\cite{carleson}, p. 70). For the Latte maps there is 0 components, for $f(z)=z^2-2$ one, for $f(z)=z^2$ two and for $f(z)=z^2-1$ infinitely many.

After this short excursion to the complex dynamics we shall return to the main subject of the article --- the local stabilization of the given dynamical system near repelling fixed point or repelling cycle.

\subsection{A-stabilization}  

Let $\mathcal{O}$  be a unstable cycle of the dynamical system $f : \hat{\mathbb{C} }\to \hat{\mathbb{C} }$. Following \cite{dm1} we use some averaging procedure to stabilize this dynamical system near $\mathcal{O}$.

The idea of stabilization procedure is well-known and simple. In Control Theory it calls \textit{feed-back control}. We measure  some previous states (\textit{trajectory}) of the given  dynamical system near given unstable cycle and then add a \textit{control} which correct the next state. The control procedure should be of course independent  of the given initial state of the system. This means that it should be the same for all initial states from the some (small enough) neighbourhood of the unstable cycle. 

To clarify the ideas we start with the  description of stabilization of given unstable  fixed point $s$ (i. e. cycle of length one).

 Let $z_0$ be a point close enough to $s$. This is an initial state.  Consider the first $n$ points of the orbit, i.e. $z_0, f(z_0),   f^2(z_0), \dots f^{n-1}(z_0)$, which we call $z_0, z_1, z_2, \dots z_{n-1}$ and this is a part of the trajectory of $z_0$.
Next point of this orbit, namely $f(z_{n-1})$ could lie far enough from the desired fixed point $s$. Actually, for $f(z) = z^2$ with the multiplier $\mu = f'(1) = 2$ we have
\begin{equation}
\label{intr_exp}
\text{dist}(f(z_{n-1}), 1) \asymp \exp(n \log 2)\text{dist}(z_0, 1)
\end{equation}
at least for $z_0$ close enough to $1$ and numbers $n \ll  \log_2( \text{dist}(z_0, 1))^{-1}$ . This means that point $s$ is unstable and now our aim is to stabilize it.

Let $\langle a \rangle = \lbrace a_1, a_2, \dots , a_{n}\rbrace$ be an averaging set of complex numbers which will be chosen later. Let $z_0$ is a starting point, which is close enough to $s$ and   $(z_0, z_1, \dots, z_{n-1})$  be the first $n$ points of orbit.

Define the \textit{new point} $z_n^*$, which we still denote as $z_n$ later, by the rule
$$
z_n^* = f(z_{n-1}) + \text{control} = a_1 f(z_{n-1}) + a_2 f(z_{n-2}) + \dots  + a_nf(z_0)
$$
in such a way, that new one $z_n^*$ lie more close to $s$ than the old one $z_n=f(z_{n-1})$.

Continue in the same way and define for $m = n +1,  n+2, \dots $ the \textit{new trajectory }by the rule
$$
z_{m} =a_1 f(z_{m-1})+a_2f(z_{m-2})+\dots +a_{n} f(z_{m-n})
$$
where all points $z_k$ are points of the new trajectory, which coincide (generally)  with the old one only at points $z_0, z_1, \dots, z_{n-1}$.
We call this stabilization process corresponding to the given set $\langle a \rangle = \lbrace a_1, a_1, \dots, a_{n}\rbrace$ the \textit{$\langle a \rangle $-stabilization}.

We would like to stress out here that $z_n$ is a \textit{new} point and not coincide with the $n$-th point of the \textit{old} orbit. It have to be clear also that the described process is \textit{not a process of averaging} of the given (old) orbit of the dynamical system, rather the \textit{process of producing} completely new orbit.

Actually, this new system is not any more a dynamical system, it is a \textit{difference equation of $n$-th order}. Philosophically speaking, a dynamical system has \textit{no memory} and hence more easy bifurcate to the chaotic regime. When we stabilize the system by the use of memorialised  coordinates it transforms to the system \textit{with memory } and demonstrate more regular local behaviour.

Suppose that we successfully achieved our goal and $z_m \rightarrow  s$  and let $s \neq  0$. Then for all $m$ big enough and $k = 0, 1, \dots,  n -1$ we have $f (z_{m-k} ) \approx f (s) = s$ and hence
$$
s\approx z_m \approx (a_1 +a_2 + \dots +a_{n}) f(s)=(a_1 +a_2 + \dots +a_{n})s
$$
It follows that necessary condition on the set $\langle a \rangle$ is
\begin{equation}
\label{int_a}
a_1 +a_2 + \dots + a_{n} = 1
\end{equation} 

Define the polynomial corresponding to the averaging set $\langle a \rangle$ by
\begin{equation}
\label{int_pol}
p(z)=a_n+ a_{n-1}z+ \dots + a_1z^{n-1}
\end{equation}

We stress out that coefficients now are in the reverse order.

Now the necessary condition (\ref{int_a})  reads as $p(1) = 1$ and we always assume that $a_1\neq 0$.
Denote the set of all such polynomials by $\mathcal{P}_n$  and by $\mathcal{P}_n^\mathbb{R}$  the set of such polynomials
with real coefficients.

For the cycle stabilization one can use the following generalization of $\langle a \rangle$-stabilization, which we call \textit{$\langle a,T \rangle$-stabilization (of $T$-cycle)}.

Let the dynamical system has an unstable (repelling) cycle $ \mathcal{O}= (s_1, s_2, \dots , s_T )$ of the length $T$. Then the $\langle a,T \rangle$-stabilization of $T$-cycle defined as follows
$$
z_m =a_1 f(z_{m-1})+a_2f(z_{m-1-T})+\dots + a_{n} f(z_{m-1-(n-1)T})
$$
with long enough initial points of trajectory to start the process.

Using time-delayed coordinates one can rise to a new (and now indeed \textit{dynamical}) system $F : \mathbb{C}^{(n-1)T+1} \to \mathbb{C}^{(n-1)T+1}$.  It can be shown that after linearisation of $F$ near the point of the cycle the stabilization problem reduced to the position of roots of the polynomial
\begin{equation}
\label{int_chi}
\chi_T(z)=z^{(n-1)T+1}-\mu p^T(z)
\end{equation}
where
$$
p(z)=a_n+ a_{n-1}z+ \dots + a_1z^{n-1}
$$
is the corresponding polynomial for the $\langle a \rangle$-set and 
$$
\mu=\mu(\mathcal{O} )=f'(s_1)f'(s_2) ... f'(s_T)
$$ is  the multiplier of the cycle.

We do not posses here an explanation of these definitions and proofs of previous statements and refer to the recent papers (\cite{dm2}, \cite{dm3}), where this matter discussed in details.

Thus from now the problem of $\langle a,T \rangle$-stabilization has a pure algebraic context -- for the given natural number $T$ and complex number $\mu$ to find some averaging set  $\langle a \rangle$  in such a way that all roots of polynomial $\chi_T(z)$ lie in the unit disk $\Delta=\lbrace z : \vert z \vert < 1 \rbrace$.

We fix the \textit{definition}  that the (holomorphical) system $f : \hat{\mathbb{C}} \to \hat{\mathbb{C}}$ admits \textit{$\langle a,T \rangle$-stabilization } of  unstable cycle $\mathcal{O}$ with multiplier $\mu(\mathcal{O} )=\mu$ if this is the case, i.e. if all roots of corresponding polynomial $\chi_T(z)$ lie in the unit disk.

Here and through the paper $T$ is the natural number which is equal to the length of the cycle and $\mu$  is the complex number which is equal to the multiplier of the cycle.

We show that for every \textit{complex} number $\mu$ (which lie outside of the unit disc, i.e. the stabilized cycle should be repelling cycle) we can choose the finite set of \textit{complex} numbers $\langle a \rangle$ in such a way, that all roots of $\chi_T(z)$ lie in the unit disc $\Delta$.

Hence every dynamical system $f : \hat{\mathbb{C}} \to \hat{\mathbb{C}}$ with  the repelling cycle $\mathcal{O}$ could be  $a$-stabilized near this cycle. This means that for any initial condition $z_0$ close enough to $s$ the new  ($ a $-stabilized) orbit $ \text{orb}_a (z_0)=(z_0, z_1, z_2  ...)$ tends to the cycle exponentially fast.

For instance, as we show later, for our simple example $f(z) = z^2$ with $\mu=2$ at the fixed (but repelling) point $s=1$ it is impossible to find $3$ numbers $a_1, a_2, a_3$ in such a way, that $\chi_1(z)$ has all it roots inside of the unit disk $\Delta$, but we can choose $4$ numbers to reach the desired aim.
Namely these numbers (not unique but in some sense the best) are
$$
a_1=2-\sqrt{2} +i \sqrt{2}, ~~~~a_2=3(\sqrt{2}-1) -i3(\sqrt{2}-1) 
$$
$$
 a_3=-2(\sqrt{2}-1)-i 2(3-2\sqrt{2}),~~~~~ a_4=i(3-2\sqrt{2})
$$
and hence dynamical system $f(z)=z^2$ became asymptotically stable at the fixed point $1$ after $\langle a\rangle$-stabilization with $\langle a \rangle = (a_1, a_2, a_3, a_4) $.

We will show in the  section \ref{sec_formalism} that we \textit{can not} choose the real $a_k$ in the case when multiplier is real and bigger than $1$.

Let now $f(z)=z^2-2$, which  has two (finite) fixed points, namely $s=-1$ with multiplier $\mu=f'(-1)=-2$ and $s=2$ with multiplier $\mu=f'(2)=4$. Hence both fixed points are repelling.

At the point $s=-1$ the multiplier  $\mu=-2$ which has the same modules as multiplier in previous example but opposite sign.  Now \textit{we can} choose $2$ \textit{real }$a_k$ for the local $a$-stabilization near the point $-1$:
$$
a_1=\sqrt{3}-1, ~~~ a_2=2-\sqrt{3}, 
$$

These examples shows that the minimal length of the set $\langle a \rangle$, which we denote by $\text{ord} \langle a \rangle$  depends not only of the \textit{magnitude} of multiplier, but rather of it \textit{position} on the complex plane.

\subsection{Stabilization domain and $p/q$-duality } 

For the given polynomial $p \in \mathcal{P} _n$  we define it \textit{stabilization domain for $T$-cycle} by 
$$ 
S^T(p)=\lbrace \mu: \text{all roots of}~~ \chi_T(z) ~~\text{lie in}~~ \Delta  \rbrace
$$

This set describe the set of values of multiplier of the given cycle of length $T$ for which this cycle could be asymptotically stabilized via chosen $\langle a, T \rangle$-stabilization process. For the given family of polynomials $\mathcal{P}$ we define it stabilization domain for $T$-cycle
by
$$
S^T(\mathcal{P})=\bigcup_{p\in \mathcal{P}} S^T(p)
$$
We denote by $S^T_n$ the stabilization domain for $T$-cycle for the family $\mathcal{P}_n$.

The set $S^T_n$ describes the largest domain of those $\mu$ for which we can find $\langle a, T \rangle$-stabilization process of the order (length) $n$ asymptotically stabilizing a given cycle of the length $T$ with the multiplier $\mu$.
Often we just omit the upper index $T$ when it equal to one. Thus, for instance, we will write $S_n$ instead of $S^1_n$.

In the section \ref{sec_S_n} we give purely geometrical description of the set $S_n$. It turns out that $S_n$ is an open set bounded by the sinusoidal spiral (we refer to \cite{yates}, p. 214 for definition) with the one punctured point inside.  Namely, for instance, $S_2$ is bounded by \textit{cardioid} and $S_3$ by  \textit{Cayley's sextic} (with the punctured point $\lbrace 1 \rbrace$). 

Since $S_n$ is a union of all $S(p)$ for $p\in \mathcal{P}_n$   it is useful to find some subfamily of simple polynomials $p_\alpha$, such that $S(p_\alpha)$ also cover $S_n$. We show that for polynomials of very special kind their stabilization domains $S(p_\alpha)$ indeed cover all set $S_n$. This gives the practical rules to design the $a$-stabilization sets.

In the section \ref{sec_formalism} we describe so called $p / q$ -duality, which allow us to reformulate the problem of roots position of $\chi_T(z)$ in $\Delta$ to the problem of image position of $q(\bar{\Delta})$, where $q(z)=z(p^*(z))^T$ and $p^*(z)$ is the inverse polynomial of $p(z)$.

This leads to the simple practical rule: \textbf{\textit{the set $ \langle a \rangle $ stabilize locally any dynamical system with the multiplier $\mu$  if and only if $q(z)=a_1z+a_2z^2+ \dots +a_nz^n$ omit point  $1/\mu$ in $\bar{\Delta}$.}}

\subsection{Stabilization of a family of maps}

Usually, for the given family of dynamical systems $f_c(z)$ parametrized by the parameter $c$ we have a partition of the parameter space into disjoint regions in such a way that dynamics of the iterates of $f_c$ in this regions display essentially the same features, while as $c$ passes from one region to another, some significant change in the dynamics (\textit{bifurcation}) take place.

For instance, for the family of quadratic polynomials $f_c(z) = z^2 + c$, the partition of the \textit{parameter plane }$\mathbb{C}$ leads to the \textit{Mandelbrot set }
$$
\mathcal{M}=\lbrace c: f^n_c(0)~~~ \text{bounded when} ~~~n \to \infty \rbrace
$$

This extremely popular mathematical object (also in the non-mathematical world) is now well understood, except probably one, but central question of the theory --- is it true that $\mathcal{M}$ is \textit{locally connected}.

We recall that a \textit{hyperbolic component} $\mathcal{H} $ of the set $\mathcal{M}$ is an open connected component of the set of parameters $c$, such that$f_c$ has an attracting  cycle $\mathcal{O}_c$.

It is known that for every $c\in \mathcal{H} $ the attracting cycle $\mathcal{O}_c$ has constant period through $\mathcal{H} $ and  $\mathcal{O}_c$ moves holomorphically through the parameter $c$ moves in $\mathcal{H} $. Thus in hyperbolic components the behaviour of $f_c(z)$ is \textit{structurally stable} or \textit{robust}. 

The question of whether there exist non-hyperbolic components is open and essentially goes back to the classical work of Pierre Fatou (1920). 

On the other hand in hyperbolic components the corresponding repelling cycles $\mathcal{O}_c$ has also a constant period through $\mathcal{H} $ and lie in the corresponding Julia sets.

Previous  observations rise to the next problem in $a$-stabilization --- to find \textit{one} $a$-stabilization process for the given \textit{domain of parameters} or, in other words, for the given \textit{domain of multipliers}.

As a simple example consider the family of quadratic polynomials $f_c(z)=z^2+c$, where a parameter $c$ lie in the one of hyperbolic components of Mandelbrot set, let say in the main cardioid. Then every $f_c$ has $2$ fixed points, one attractive and one repelling. At the repelling fixed point as parameter moves through the main cardioid, the corresponding multiplier describe the set $M=\lbrace z : \vert z-2 \vert <1\rbrace$ which is a circle. Thus to stabilize every dynamical system from this hyperbolic component by the one stabilization process (or in other words by the one set $\langle a  \rangle$ ) we have to find one polynomial $p(z)$, such that $M \subseteq S(p)$.

Now by the $p/q$-duality we reduce this question to the problem of finding a polynomial $q(z)$ of the smallest degree, such that $q(0)=0$, $q(1)=1$ and $q(z)$ omit the set $M^*=\lbrace w: 1/w \in M\rbrace$  when $z\in \bar{\Delta}$. It is easy to see that in this particular case $M^*=\lbrace w: \vert w-2/3 \vert<1/3 \rbrace$  is a circle. 

In other words we have to find a polynomial $q(z)$, such that $q(\bar{\Delta})\subseteq \Omega_M$, where $\Omega_M=\hat{\mathbb{C} }\setminus M^*$, which is the sphere with the hole. We shall to stress out that in this particular case the set $\Omega_M$ is \textit{closed} (but not a simply connected). Hence there also no obstructions to find the extremal polynomial in question.  

After we find polynomial $q(z)$ it coefficients  give us the desired set $\langle a \rangle$ which stabilise every system $f_c(z)=z^2+c$ at the repelling fixed point for every value of parameter $c$ from the main cardioid.  This how it works and now we consider the general situation.

Let $M$ be a subset of the complex plane and $ \lbrace f_c \rbrace_{c\in \mathcal{X}}$ be a given family of dynamical systems, such that every $f_c$ has an unstable cycle $\mathcal{O}_c=(s_c, f_c(s_c), f_c^2(s_c), ... f_c^{T-1}(s_c))$ with the multiplier of the cycle $\mu ( \mathcal{O}_c)\in M$. 

We shall to find one polynomial $p(z)$ of the smallest degree, such that $M\subseteq S^T(p)$. Coefficients of this polynomial is the corresponding set $\langle a \rangle$ which $\langle a, T \rangle $-stabilize every unstable cycle $\mathcal{O}_c$.

According to the $p/q$ -- duality this problem equivalent to the problem of description of the range $q(\bar{\Delta})$ for polynomials $q(z)$ satisfied some conditions, namely to find polynomial  $q(z)=z(p^*(z))^T$ of \textit{the smallest degree,} such that $q(1)=1$ and $q(z)$ omit in the closed unit disc some prescribed set of values $M^*$.

Our approach to the solution be as follows. First we show that the range restrictions implies estimates of (Taylor) coefficients of polynomial $q(z)$. This is a general principle of Geometric Function Theory. From these estimates and normalization condition $q(1)=1$ we can get  an  estimates \textit{from below } of the degree of the polynomial $q(z)$.

For simply connected domain $\Omega_M=\mathbb{C}\setminus M^*$ we shall give also the different approach based on the subordination principle and growth estimates near the boundary of the Riemann function $g_{\Omega_M} : \Delta \to \Omega_M$
conformally mapping $\Delta$ onto $\Omega_M$. 

This approach lead to the degree estimate from below for polynomial $q(z)$ for \textit{any simply connected } set $\Omega_M$.

In order to find concrete sequence $\langle a \rangle $ which stabilize the given set of unstable cycles, we have to find polynomial $q(z)$ of the smallest degree which omit prescribed set of values in $\bar{\Delta} $, or in other words it image $q(\bar{\Delta})$  lie in the prescribed set $\Omega_M$

There is a some different  approaches to solve this problem. 

First one is based on some classical results of Ted Suffridge (\cite{suffridge} ) and can be applied for slit domains $\Omega_M$. This lead to the extremal $\langle a,1 \rangle $ and $\langle a, 2 \rangle $ stabilization sets for real systems (see also \cite{dm1} and \cite{dm2}).

Second one is based on the  result of V. Andrievski and S. Ruscheweyh (\cite{andru} from  Approximation Theory, which we call Theorem A.  We use Theorem A to construct polynomial which subordinate to $g_{\Omega_M}(z)$ and superordinate  to $g_{\Omega_M}((1-\frac{c}{n})  z)$.  After normalization this polynomial became the desired polynomial $q(z)$.

Third way based on the theory of maximal range of A.  Cordova and S. Ruscheweyh (\cite{corumr1},  \cite{corumr2}) and give rise to the extremal polynomials in question. But this approach  has disadvantage that it works only for very special kinds of domains $\Omega_M$. At least is perfectly works for slit domains and circular domains.   But the extremal solutions are of a little interest in the questions of stabilization since they are very sensitive to the averaging coefficients choose --- they could loose their stabilization properties after a little change in coefficients. This why we do not consider in these notes extremal solutions.

We now briefly summarize some concrete results. 

Let  $ \lbrace f_c \rbrace_{c\in \mathcal{X}}$ be a given family of dynamical systems, such that every $f_c$ has an unstable cycle $\mathcal{O}_c=(s_c, f_c(s_c), f_c^2(s_c), ... f_c^{T-1}(s_c))$ of length $T$ with the multiplier of the cycle $\mu ( \mathcal{O}_c)\in M$, where $M$ is a given subset of the complex plane $\mathbb{C}$.

Denote by $\mu_M=\sup \lbrace \vert z  \vert : z \in M \rbrace$ the \textit{size  } of the multipliers set $M$. It is clear that the length of the stabilization sequence $\langle a \rangle$  have to depend of the size of $M$. But it also deeply depends of the \textit{shape} of $M$ and almost not depend of the length of the cycle.

For the \textit{one point} set $M=\lbrace \mu \rbrace$,
where $\vert \mu \vert \geq  1$ and  $\mu \neq 1$ we can  always  find  a stabilization set $ \langle a \rangle$ of the order 
\begin{equation}
\label{log}
\text{ord} \langle a \rangle  \asymp \log \mu_M
\end{equation}
and this estimate is the best possible in  order.

For the real segment  $M=\lbrace z : \Im z =0, -\mu_M \leq \Re z \leq -1 \rbrace$, where  $\mu_M>1$ we can always find the \textit{real} stabilization set $\langle a \rangle$, such that 
\begin{equation}
\label{sq}
\text{ord} \langle a \rangle  \asymp \sqrt{\mu_M}
\end{equation}
and again this estimate is the best possible in order.

For the (left) horocycle  $M=\lbrace z : \vert z+\mu_M/2 \vert \leq \mu_M/2$ we can always find the real stabilization set $\langle a \rangle$, such that
\begin{equation}
\label{one}
\text{ord} \langle a \rangle  \asymp \mu_M
\end{equation}
and again this estimate is the best possible in order.

Let now  $M=\lbrace z ; \vert z \vert \leq \mu_M, | \arg \mu | \geq \theta \rbrace $ is a (wide) sector of a (big) circle. This set $M$ (in some sense) is  the largest domain of admissible multipliers, since the necessary condition on the set $M$ to satisfy the inclusion $ M \subseteq  S^T(p) $ is that the set  $\hat{\mathbb{C}}\setminus M$ have to contain some open subset  connected points $1$ and $\infty$ on the Riemann sphere $\hat{\mathbb{C}}$.

As we show later for this set $M$ we can find a real stabilization set  $\langle a \rangle$, such that 
\begin{equation}
\label{exp}
\text{ord} \langle a \rangle  \asymp c(\theta) \exp \mu_M
\end{equation}
and the estimate is the best possible in order. We stress out that even for  $\theta=\pi/2$ we can get only  exponential order of the length of the stabilization set.

We shall to point out  that in this notes we do not use  any specific properties of complex dynamical systems and all this results are true as for a complex as well for a real dynamical systems. The only difference is that for the real systems the set of multipliers is also a \textit{real}  subset of the complex plane as well as a stabilized sequence $\langle a \rangle$.

In the forthcoming article we give an applications of previous results and describe the global dynamics of two typical complex dynamical systems $f_0(z)=z^2$ and $f_{-2}(z)=z^2-2$ after $a$-stabilization. These examples corresponds to the 'center' point $c=0$ of the Mandelbrot set $\mathcal{M}$  with $J(f_0)=\partial \Delta$ and to the (left) extreme point $c=-2$ of $\mathcal{M}$ with $J(f_{-2})=[-2,2]$.

For the family of real dynamical systems $g_\lambda(x)$, mentioned at the beginning, this two examples  corresponds by the formula
 $$
 c=\frac{\lambda}{2}(1-\frac{\lambda}{2})
 $$
to the \textit{main trunk} ($c=0$ or $\lambda=2$) of bifurcation diagram, where dynamics is \textit{completely regular} and to the \textit{top of the bifurcation tree} ($c=-2$ or $\lambda=4$), where dynamics is \textit{completely chaotic}.

\section{The set $S_n$ }  
\label{sec_S_n}
We recall, that $S_n$ is the set of the  multipliers  at the unstable fixed point such that  we can $\langle a \rangle$-stabilise the system  near this fixed point by the averaging process of length at most $n$. Next theorem give the explicit formula for this set.

Here and through the paper we denote by $g(E)$ the image of the set $E$ by the map $g$, i.e. $g(E)=\lbrace g(z) : z \in E \rbrace$.
\begin{theorem}
\label{theo_1}
$S_n=u_n(\Delta)$, where $u_n(z)=1-(1-z)^n$
\end{theorem}  
\begin{proof}
Let $\mu \in S_n$, which means that there is a polynomial $p(z)$ of degree not greater than $n-1$, such that $p(1)=1$ and all roots of 
\begin{equation}
\label{sec1_chi}
\chi(z)=z^n-\mu p(z)
\end{equation}
lie in the unit circle.
Denote these roots by $\zeta_1, \zeta_2, \dots  \zeta_n$. Then
\begin{equation}
\label{sec1_chi_roots}
\chi(z) =(z-\zeta_1)(z-\zeta_2)\dots (z-\zeta_n)=z^n-\sum_{k=1}^n(-1)^{k-1}\sigma_k(\zeta)z^{n-k}
\end{equation}
where
\begin{equation}
\label{sec1_sigma}
\sigma_k(\zeta)=\sigma_k (\zeta_1,\zeta_2, \dots \zeta_n)=\sum_{1 \leq i_1<i_2< \dots i_k \leq n } \zeta_1 \zeta_2 \dots \zeta_k
\end{equation}
are an elementary symmetric polynomials. Define
$$
U_n(\zeta)=\sum_{k=1}^n(-1)^{k-1}\sigma_k(\zeta)
$$
in the poly-disc $\Delta^n=\Delta \times \Delta\times \dots \Delta$.

Since $p(1)=1$, from  (\ref{sec1_chi}) and (\ref{sec1_chi_roots}) for $z=1$ we conclude that
\begin{equation}
\label{sec1_incl1}
S_n\subseteq U_n(\Delta^n)
\end{equation}

On the other hand, let  $\mu \in U_n(\Delta^n)$, which means that there is a $\zeta \in \Delta^n$ such that $\mu=U_n(\zeta)$. We can assume that $\mu \neq 0$.  Define for $k=1, 2\dots,  n$ complex numbers $a_k$ by the rule
\begin{equation}
\label{sec1_a_k}
a_k=\frac{(-1)^k}{-\mu}\sigma_k (\zeta)
\end{equation}
and let 
$$
p(z)=a_n+a_{n-1}z+\dots +a_1z^{n-1}=\sum_{k=1}^{n}\frac{(-1)^k}{-\mu}\sigma_k (\zeta)z^{n-k}
$$ Then $p(1)=1$ and
$$
\chi(z)=z^n-\mu p(z)=z^n+\sum_{k=1}^n(-1)^{k}\sigma_k(\zeta)z^{n-k}=(z-\zeta_1)(z-\zeta_2)\dots (z-\zeta_n)
$$
Hence all roots of $\chi(z)$ lie in $\Delta$. This implies
\begin{equation}
\label{sec_1_incl2}
U_n(\Delta^n)\subseteq S_n
\end{equation}
From (\ref{sec1_incl1}) and (\ref{sec_1_incl2}) we conclude that
\begin{equation}
\label{sec_1_eq}
U_n(\Delta^n)=S_n
\end{equation}
Now we apply the very special case of Laguerre Theorem (or Grace apolarity theorem or Walsh theorem)  (\cite{sheilsmall}, p. 182), which state that 
$$
U_n(\Delta^n)=U_n(\text{diag}\Delta^n)
$$
But $U_n(z,z,\dots,z)=u_n(z)$ and hence $U_n(\text{diag}\Delta^n)=u_n(\Delta)$. 
\end{proof}

Now we can describe the set $S_n$ in a purely geometric terms: $S_n$ is an open set bounded by the curve
\begin{equation}
z(\varphi)=1- 2^n\cos^n(\frac{\varphi}{n})e^{i\varphi}
\end{equation} 
for $-\pi \leq \varphi \leq \pi$ and with punctured point $\lbrace 1 \rbrace$.

Indeed, let $\Delta^1=\lbrace z: \vert z -1 \vert < 1\rbrace$ be the shifted unit disk. Then 
\begin{equation}
\label{sec_1_cor_S_n_view}
u_n(\Delta)=s( r (m_n (\Delta^1)))
\end{equation}
where $s(z)=z+1$ is shift, $r(z)=-z$ is reflection and $m_n(z)=z^n$ .

Hence $S_n$ is an open set, since $u_n$ is an open mapping. 

We claim, that $u_n(\Delta)$ is starlike with respect to $\lbrace 1 \rbrace$, since $\Delta^1$ is starlike with respect to $\lbrace 0 \rbrace$. Hence we shall concentrate only on the  image of the boundary $\partial \Delta^1$.

Denote point on the boundary $\zeta = \rho (\theta) e^{i \theta}$, where $\rho(\theta)=2\cos \theta$ and $-\frac{\pi}{2} \leq \theta < \frac{\pi}{2}$. It follows that $ m_n(\partial \Delta^1)$ is the curve $2^n \cos^n(\theta) e^{in\theta}$ and after substitution $\varphi=n\theta$ the curve $z_0(\varphi)=2^n\cos^n(\varphi/n)e^{i\varphi}$
for $-\frac{n\pi}{2} \leq \varphi < \frac{n\pi}{2}$.

But, because every segment $\langle 0, \rho(\frac{\varphi}{n}) e^{i\frac{\varphi}{n}}\rangle $ maps by the function $m_n$ to the corresponding segment $\langle 0,z_0(\varphi) \rangle $ we claim, that $m_n(\Delta^1)$ bounded only by the part of the curve $z_0(\phi)$ corresponding to $-\pi \leq \phi <\pi$.

Now the statement of the proposition follows from  (\ref{sec_1_cor_S_n_view})

We conclude this section pointed the interesting relations between boundary of $S_n$ and caustics.

Caustic in the optical problems is the place of the concentration of light.

To be more precise, let $S$  be a given curve and let $P$ be a fixed point called the radiant point. If rays
from $P$ are reflected by the curve, the envelope of the reflected rays is called the
\textit{caustic} of $S$ with $P$ as radiant point. Expressing this geometrically, let any line
through $P$  meet the curve at $Q$ and let $ QA$ be drawn so that $QA$ and $QP$ make
equal angles with the tangent to the curve at $Q$; then the envelope of $QA$ is the
\textit{caustic}.

This caustic by reflection is sometimes called the \textit{catacaustic}, to distinguish it
from a curve similarly formed (the \textit{diacaustic)} when the rays are refracted.

You can see the caustic in the cup of coffee in the morning. It looks like cardioid but it is actually the \textit{nephroid}. It is not surprising that a caustic appear in our notes, because usually they are  discriminant sets for different families(or bifurcation sets for dynamical systems) .

Caustics were first introduced and studied by
Tschirnhausen in 1682. Other contributors were Huygens,
Quetelet, Lagrange, and Cayley (\cite{yates}, p. 15).

Let $S$ be a given curve and $O$ be a given point. With center at any
point $A$ of $S$, and radius $AO$, draw a circle: the envelope of such circles is called the\textit{ orthotomic
of $S$ with respect to $O$} (\cite{lock} , p. 153). The caustic is the evolute of the \textit{orthotomic}.

It is easy to see that $S_2$ bounded by \textit{cardioid}, first studied by Roemer (1674) in an investigation for the best form of gear teeth (\cite{yates}, p. 4) . The name cardioid (''heart-shaped'') was first used by de Castillon (1741)  (\cite{lock} , p. 43).

To draw the border of $S_2$, start  from the unit circle with the center at the origin and mark by $O$ the point $1$. With center at any point $B$ on the circle, and radius $BO$, draw another circle. Repeat for a large number of positions of $B$, spread evenly round the base-circle. The heart-shaped curve which all these circles touch is the \textit{cardioid}. The pointed part at $O$ is called a \textit{cusp}.

Thus cardioid is the orthotomic of the unit circle (with respect to the $O=\lbrace 1 \rbrace$).

Cardioid is also the \textit{caustic} of the circle and generated by the point \textit{on the circle} (the source of light)  after reflection by this circle.

Boundary of $S_3$  is \textit{Cayley's sextic}, first studied by Tschirnhausen and in the deep details by Arthur Cayley. This curve is an \textit{orthotomic curve }  (or secondary caustic)  of cardioid .

It is not difficult to see, that each next border of $S_n$ is the orthotomic curve of  the previous one border $S_{n-1}$.

\section{The simplest polynomials}

Let $\mu \in S_n$. Then, according to Theorem \ref{theo_1},  there exist  $\zeta=\zeta(\mu)\in \Delta$ such that $\mu= 1-(1-\zeta)^n$.  Define the (\textit{simplest})  polynomials $p^{\mu, \zeta}_n(z)=\frac{1}{\mu}(z^n-(z-\zeta )^n)$. It is clear, that $\deg(p^{\mu, \zeta}_n) =  n-1$. and  $p^{\mu, \zeta}_n(1)=1$. Hence $p^{\mu, \zeta}_n(z)\in \mathcal{P}_n$ and since the only root of 
$$
\chi(z)=z^n-\mu p^{\mu, \zeta}_n (z)=(z-\zeta)^n
$$
is $\zeta$ we have that $\mu \in S(p^{\mu, \zeta}_n)$.

For every $\mu\in S_n$ denote by $\mathcal{P}_n(\mu)$ the (finite) set of simplest polynomials. Then we have that $\mathcal{P}_n(\mu) \subset \mathcal{P}_n$. 
Since $S_n$ is a union of all $S(p)$ for $p\in \mathcal{P}_n$ we see that
$$
\bigcup_{\mu\in S_n, p\in \mathcal{P}_n(\mu)} S(p)\subseteq S_n
$$

On the other hand previous observations showed that for every $\mu\in S_n$ we can choose $p\in \mathcal{P}_n(\mu)$ such that $\mu\in S(p)$.

This implies that

\begin{equation}
\label{prop_p_union}
S_n=\bigcup_{\mu \in S_n, p\in \mathcal{P}_n(\mu) } S(p)
\end{equation}
We shall to stress out that for the given $\mu \in S_n$ there could be $1,  2, \dots , n-1$ simplest polynomials. This depends of the $\mu$ position in $S_n$. For instance for $n=2$ there is always $1$ simplest polynomial. For $ n=3$ for some $\mu$ it could be $2$ simplest polynomials, but mostly only $1$ and so force.
We shall also to point out that $\deg p^{\mu, \zeta}_n =n-1$, but the number of elements of the  corresponding set $\langle a \rangle$ is $n$.

Previous proposition gives us a practical rule to construct the $a$-stabilization sets for different $\mu$. Let we consider two examples mentioned in Introduction.

\begin{example}
\label{example_z^2}
\end{example} Let $f(z)=z^2$ and we would like to $\langle a \rangle$-stabilize the given system at the fixed point $1$. Then, since $\mu=f'(1)=2$ and $2\notin S_3$ (point $2$ lie exactly at the boundary of $S_3$), we can not choose $3$ numbers $a_1, a_2, a_3$ in such a way, that corresponding polynomial $\chi(z)$ has all roots inside the unit disc $\Delta$.

But, according to theorem \ref{theo_1} we can choose $4$ numbers, since $2 \in S_4$. Now proposition (\ref{prop_p_union})  implies that we can choose polynomial $p^{2, \zeta}_4(z)$ with $2 \in S(p)$.

According to previous general observations we have to choose $\zeta \in \Delta$ in such a way, that $2=1-(1-\zeta)^4$ or $\zeta=1-\omega^k_4\omega^{1/2}_4$ for some $k=1,2,3,4$, where $\omega_4=i$  and $\omega^{1/2}_4=e^{i\frac{\pi}{4}}$.  But it is easy to see, that only $2$ of these $\zeta$ lie in $\Delta$  ( for $k=1$ and $k=4$).  

Hence we can take $\zeta_4= 1-\frac{\sqrt{2}}{2} + i \frac{\sqrt{2}}{2} $ and  a little calculation shows that $p^{2, \zeta_4}_4(z)=a_4+a_3z+a_2z^2+a_1z^3$ with $a_4=i(3-2\sqrt{2}), a_3=2(\sqrt{2}-1)-i 2(3-2\sqrt{2}), a_2=3(\sqrt{2}-1) -i3(\sqrt{2}-1)$   and  $ a_1=2-\sqrt{2} +i \sqrt{2}$. 

The desired stabilization process  started from the given $z_0$ close enough to $1$, $z_1=z_0^2$, $z_2=z_1^2$,  $z_3=z_2^2$ and continue for $m=3,4,\dots$ in the following way:
$$
z_{m+1}=a_1f(z_m)+ a_2f(z_{m-1}) +a_3 f(z_{m-2})+ a_4f(z_{m-3})
$$

\begin{example} 
\end{example}
Let  $f(z)=z^2-2$. Then equation $z^2-2=z$ has two roots, namely $s_1= -1$ and $s_2=2$. Hence this dynamical system has two fixed points: $f(-1)=-1$ and $f(2)=2$. Since both multipliers $\mu=f'(-1)=-2$ and $\nu=f'(2)=4 $  lie outside of $\bar{\Delta}$, both fixed points are  repelling.

Let we shall to stabilize this dynamical system near repelling point $-1$. Since 
$\mu\in S_2$ we can choose $2$ coefficients $a_1$ and $a_2$, i. e. polynomial $p^{-2, \zeta}_2(z) $ of the first degree. 

We have to choose $\zeta \in \Delta$ in such a way, that $-2=1-(1-\zeta)^2$ or $\zeta=1\mp \sqrt{3} $.  But it is easy to see, that only $\zeta=1-\sqrt{3}$ lie in $\Delta$.

Hence $p^{-2, \zeta}_2(z) = - \frac{1}{2}(z^2-(z-\zeta)^2) = (\sqrt{3}-1)z+2-\sqrt{3}$ and stabilization coefficients are $a_1=\sqrt{3}-1$ and $a_2=2-\sqrt{3}$.

Chose the starting point $z_0$ close enough to $-1$. Actually it can be chosen everywhere on the real interval $(-2,2)$ or inside of  the circle $z: |z+1|<1/6$. Define $z_1=z_0^2-2$.

Now for $m=1,2,3, \dots $  define
$$
z_{m+1}=a_1f(z_{m})+a_2f(z_{m-1})=(\sqrt{3}-1)(z_m^2-2)+(2-\sqrt{3})(z_{m-1}^2-2)=
$$
$$
(\sqrt{3}-1)z_m^2+(2-\sqrt{3})z_{m-1}^2-2
$$
and we can see that according to the general theory $z_m\to -1$ .

\section{On the $p / q$ duality}
\label{sec_formalism}

We call by  '$p / q$ \textit{duality}' the elementary  passage from the problem of  roots positions of $\chi_T(z)$ to the $q(\bar{\Delta})$ description for some polynomial $q(z)$. It is based on the simple fact that the problem of stability of the given polynomial $\chi(z)$ of degree $n$  (i. e. all roots of $\chi(z$) lie in $\Delta$ ) is equivalent to the problem of the $\chi^*(\bar{\Delta})$ description, where $\chi^*$ is a $n$-inverse polynomial.

To be more precise we write this observation like a simple lemma.

For the given polynomial 
$$
h(z)=h_0 + h_1 z + \dots + h_n z^n
$$
we define it \textit{(multiplicative) $n$-inverse} polynomial by the rule
$$
h^*(z)=h_n + h_{n-1} z + \dots + h_0 z^n
$$

This definition is a little different with the classical one definition of \textit{conjugate polynomial}, (see e.g. \cite{sheilsmall} , p. 152 ),  but it our problems it works good and is a bit clear.

We shall use this definition even in the case when $h_n=0$. 
For example the $4$-inverse of the polynomial $z$ is $z^3$ and vice versa.

Such defined $n$-inversion is an \textit{involution } on the set of polynomials of the degree at most $n$, i. e. $h^{**}(z)=h(z)$.
 For a point $z \in \mathbb{C} $ denote by $ z^*=1/z$ a \textit{multiplicative inverse} of $z$ and define $ (\infty)^*=0$ and $0^*=\infty$.  
 
Denote $E^*=\lbrace z^* \colon z \in E \rbrace$ , $E ^T=\lbrace z^T: z \in E \rbrace$ and $E ^\frac{1}{T}=\lbrace z: z^T \in E \rbrace$ where $T$ is a given natural number corresponding to the length of the cycle. Operation $^*$ is one to one mapping of $\hat{\mathbb{C}} $ and hence commutes with all ordinary set operations, i. e. $(A\cup B )^*=A^*\cup B^*$, $(A\cap B )^*=A^*\cap B^*$ and $(\hat{\mathbb{C}}\setminus A)^*=\hat{\mathbb{C}}\setminus A^*$

\begin{lemma} 
\label{lemma_1}  Let $\chi (z)= c_0 + c_1z + \dots+ c_n z ^n$  be a polynomial with $c_n \neq 0$. Then, 
\begin{equation}
\label{g1}
\mbox{all zeros of} ~~~~~\chi (z)~~~~~~ \mbox{ lie in }~~~~ \Delta
\end{equation}
if and only if 
\begin{equation}
\label{g2}
0 \notin \chi^*(\bar{\Delta})
\end{equation}
where 
\begin{equation}
\label{chi}
\chi ^* (z)=c_n + c_{n-1}z + \dots +c_0z^n
\end{equation}

\end{lemma}
\begin{proof}
Let $\chi(z)=c_nz^m(z-z_{m+1} )\dots (z-z_n)$, where $0 \leq m<n$ (the case $m=n$ is trivial). Then \begin{equation}
\label{g3}
\chi ^* (z) = z^n \chi (1/z)=c_n(1-z_{m+1}z)\dots (1-z_nz)
\end{equation}.

Thus, from (\ref{g1}) we conclude that all zeros of $\chi^* (z) $ lie outside of $\bar{\Delta}$, which implies (\ref{g2}).
On the other hand, if (\ref{g2}) holds, then (\ref{g3})    implies (\ref{g1}).
\end{proof}

\begin{remark} We can not remove the condition $c_n \neq 0$. Indeed, let say $n=4$ and $\chi (z)=z$. Then all zeros of $\chi (z)$ lie in $\Delta$, but  for $\chi^*(z)=z^3$ we have $0 \in \chi^*( \bar{\Delta})$.
\end{remark}

Next theorem express the stability domain $S^T(p)$ of the given polynomial $p(z)$ it terms of the set $q(\bar{\Delta})$ for the corresponding $q$-polynomial. This we call the $p/q$-duality.

\begin{theorem}
\label{theo2}
 Let $\mathcal{O}$  be a unstable cycle of the length $T$ of the dynamical system $f : \hat{\mathbb{C} }\to \hat{\mathbb{C} }$ with multiplier of the cycle $\mu(\mathcal{O})=\mu$.
 
 Then the set of complex numbers $\langle a \rangle=\lbrace a_1, a_2. \dots , a_n \rbrace$ stabilize locally this cycle  if and only if polynomial
 $$
 q_{1/T}(z)=a_1z^{T+1}+a_2z^{2T+1}+ \dots + a_nz^{nT+1}
 $$
 omit values $\mu^{-\frac{1}{T}}$ in the closed unit disc $\bar{\Delta}$.
 \end{theorem}

\begin{proof} Let
$$
p(z)=a_n + a_{n-1}z + \dots + a_1 z^{n-1}
$$
and denote $ q(z)=z(p^*(z))^T $ where $p^*(z)$ is the $(n-1)$-inverse of $p(z)$. We also denote $q_{1/T}(z)=zp^*(z^T)$  the  $T$-root transform of the polynomial $q(z)$.

We can express the statement of the theorem as
\begin{equation}
\label{sch2}
S^T(p)=(\hat{\mathbb{C} }\setminus q(\bar{\Delta}))^*
\end{equation}
or, in different terms,
\begin{equation}
\label{sch3}
S^T(p)=(\hat{\mathbb{C} }\setminus (q_{\frac{1}{T}}(\bar{\Delta}))^T)^*
\end{equation}
 
Let 
\begin{equation}
\label{hi}
\chi (z)=z^{(n-1)T+1} - \mu (p(z))^T
\end{equation}
We can assume that $\mu \neq 0$.

It is clear (use  (\ref{g3})) that $((n-1)T+1)$ --inverse of $\chi(z)$  is 
\begin{equation}
\label{d2}
\chi^*(z)=1- z\mu (p^*(z))^T
\end{equation}
where $p^*(z)$ is the $(n-1)$-inverse polynomial of $p$. Hence $\chi^*(z)=0$ if and only if $q(z)=\mu^*$ and

\begin{equation}
\label{d3}
0\notin \chi^*(\bar{\Delta} ) \Leftrightarrow \mu ^* \notin q(\bar{\Delta})
\end{equation}
Now  (\ref{sch2}) follows from the definition of $S^T(p)$ and Lemma  \ref{lemma_1}.

In order to prove (\ref{sch3}) we have only to claim that
$$
q(\bar{\Delta})=(q_\frac{1}{T}(\bar{\Delta}))^T
$$
\end{proof}

It is much easer to work with the polynomial $q_\frac{1}{T}(z)$ instead of $q(z)$, since it  spectrum lie in the arithmetic progression $T\mathbb{Z} +1$ and the degree is the same as degree of $q(z)$. 

Functions $g(z)$ regular  in $\Delta$ with the  spectrum in the arithmetic progression $T\mathbb{Z} +1$, i.e such that
\begin{equation}
\label{sym1}
g(z)= \hat{g}(1) z + \hat{g}(T+1)z^{T+1} + \hat{g}(2T+1) z^{2T+1} + \hat{g}(3T+1) z^{3T+1} +   \dots 
\end{equation}
are called \textit{ $T$-symmetric} ( see \cite{hayman}, p. 95), since the image of the unit disc $g(\Delta)$ has  $T$-fold rotational symmetry, that is invariant under the rotation $w\rightarrow \omega_T  w$, where $\omega_T=e^\frac{2\pi i}{T}$ is the $T$-root of unity (see Appendix  for the proof).

\begin{remark} We shall to point out that the stability set $S^T(p)$ in general has not $T$-fold rotational symmetry.

For example, for $p(z)=1+z$ and $T=2$ the corresponding stability set $S^2(p)$  is $(\hat{\mathbb{C} }\setminus q(\bar{\Delta}))^*$ , where $q(z)=z(1+z)^2$ and this set has not rotational symmetry as well as the set $q(\bar{\Delta})$.

Meanwhile for $q_\frac{1}{2}(z)=z(1+z^2)$ corresponding set $q_\frac{1}{2}(\bar{\Delta})$ has $2$-fold rotational symmetry, but after we "squaring" the set $q_{\frac{1}{2}}(\bar{\Delta}) $ this symmetry disappear.
\end{remark}

The next corollary is a simple application of the previous theorem.

\begin{corollary}

Let  $\mu \geq 1$  be a real number. Then we could not find a real set $\langle a \rangle $ which stabilize a cycle of length $T$ with the multiplier equal to $\mu$.

\end{corollary}
\begin{proof}
Let 
$$
q_{1/T} (z)=zp^*(z^T)=a_1z^{T+1}+a_2z^{2T+1}+ \dots + a_nz^{nT+1}
$$ 
and all $a_k$ are real. Then $q_{1/T}(x)$ is real for real $x$ and, since $q(0)=0$ and $q(1)=1$ we have that all real segment $[0,1]$ covered by the values of $q_{1/T} (z)$ in $\bar{\Delta}$ and hence if $\mu \geq 1$ and real  $q_{1/T} (z)$ can not omit value $\mu^{-1/T}$ in $\bar{\Delta}$. Now Theorem \ref{theo2} implies that the set $\langle a \rangle$ could not stabilize the cycle with multiplier $\mu$.
\end{proof}

\begin{remark} We will show later that for any $\mu$, which is not real and bigger than $1$ we can find real stabilisation set $\langle a \rangle$ which successfully $\langle a, T \rangle$ stabilize cycles with multiplier $\mu$.

\end{remark}

\section{Stabilisation for the given set of multipliers}

In this section we describe the strategy to find  \textit{one} $a$-stabilization process for the given  \textit{set of multipliers} $M$ \textit{for} $T$-\textit{cycles} which stabilize locally \textit{every} $T$-cycle with multiplier in $M$.

Fix natural $T$ (length of the cycle) and let $M$ be a given  bounded subset of the complex plane $\mathbb{C}$ which we consider as the domain of multipliers of the $T$-cycles. Denote by $\mu_M= \sup\lbrace |z| : z\in M \rbrace$ the size of the set $M$.

Let $\mathcal{O}_c=(s_1(c),s_2(c), \dots s_T(c))$ be an unstable cycle  of  the map $f_c$ from the given family $\lbrace f_c \rbrace_{c\in \mathcal{X}}$ and $\mu(\mathcal{O}_c )\in M$. We shall to find one polynomial $p\in \mathcal{P}_n$ of the \textit{smallest degree}, such that $M\subseteq S^T(p)$. 

The latter inclusion means that  $\langle a,T \rangle$ -averaging process, defined by the polynomial $p$ stabilize (locally) every unstable $T$-cycle of any map $f_c$ from the given family.

Let  $\Omega_M=\hat{\mathbb{C}}\setminus M^*$ and $(\Omega_M)^\frac{1}{T}=\lbrace z: z^T \in \Omega_M \rbrace$. 
Theorem  \ref{theo2} implies that for every polynomial $p\in \mathcal{P}_n$ and any natural $T$
\begin{equation}
\label{main}
M\subseteq S^T(p)
\end{equation}
if and only if~~~~ 
\begin{equation}
\label{main_q}
q_\frac{1}{T}(\bar{\Delta})\subseteq (\Omega_M)^\frac{1}{T}
\end{equation}
where $q_\frac{1}{T}(z)=zp^*(z^T)$.

Since $q_\frac{1}{T}(0)=0$ and $q_\frac{1}{T}(1)=1$ to satisfy (\ref{main_q}) the set $(\Omega_M)^\frac{1}{T}$  have to contain points $0$ and $T$ roots of unity $\omega_T^k=e^{2 \pi i \frac{k}{T}}$ for  $k=0, 1, 2, \dots T-1$. Actually $(\Omega_M)^\frac{1}{T}$ is a $T$-symmetrical set.

 Denote by $\lambda_\Omega=\inf \lbrace |z| : z \in \Omega \rbrace$ the distance from $\mathbb{C}\setminus \Omega$ to the origin.

It is clear that for the given set $M$ and corresponding set $\Omega=(\Omega_M)^{1/T}$ we have
\begin{equation}
\label{dict1}
\mu_M=\lambda^{-T}_\Omega
\end{equation}

Denote $\mathcal{Q}^T_n=\lbrace q(z):  q(z)=zp^*(z^T),\text{where}~~~~ p\in \mathcal{P}_n\rbrace$. Thus polynomials $q\in \mathcal{Q}^T_n$ are of the form
\begin{equation}
\label{form}
q(z)=q_1z+ q_2z^{T+1}+ \dots +q_{n}z^{nT-T+1}
\end{equation}
and
\begin{equation}
\label{1}
q(1)=1
\end{equation}
We shall write $\mathcal{Q}_n$ instead of $\mathcal{Q}^1_n$.

Previous observations lead to the following problem from the Geometric Function Theory.

\textbf{Problem}. For the given $T$-symmetrical set $\Omega$, with $0\in \Omega$ and  $1\in \Omega$ to find a polynomial $q\in \mathcal{Q}^T_n$ (of the possible smallest degree), such that
\begin{equation}
\label{problem}
q(\bar{\Delta}) \subseteq \Omega
\end{equation}

\section{Necessary conditions}

First we obtain a statements of the following type. Let $\Omega$ be a given $T$-symmetrical domain. Then
\begin{equation}
\label{nes_cond_lambda}
\text{for }~~~ q\in \mathcal{Q}_n^T ~~~\text{inclusion} ~~~(\ref{problem}) ~~~\text{implies that} ~~~ n\geq \phi (\lambda_{\Omega}^{-T})
\end{equation}
for some $\phi(x) \to \infty~~~ (x \to \infty)$,    which we call \textit{necessary conditions}.

Let now $M$ be a given set of multipliers and $T$ be a length of the cycle. Then, in view of (\ref{main}), (\ref{main_q}) and (\ref{dict1}) we can rewrite (\ref{nes_cond_lambda}) in the following way
\begin{equation}
\label{nes_cond_mu}
\text{for }~~~ p\in \mathcal{P}_n ~~~\text{inclusion} ~~~M\subseteq S^T(p) ~~~\text{implies that} ~~~ n\geq \phi (\mu_M)
\end{equation}

Since polynomial map $q : \mathbb{C} \to \mathbb{C} $ is an open mapping, i .e. maps open subsets to the open one, we have that (\ref{problem}) implies
\begin{equation}
\label{open}
q(\Delta) \subseteq \Omega^o
\end{equation}
where $\Omega^o$ is the set of interior points of $\Omega$.
Hence to obtain the \textit{necessary conditions } we can always assume (\ref{open}).

First approach to the necessary conditions based on the coefficient estimates of the corresponding Riemann function. Different approach, which is based on the estimate of maximum modulus of Riemann function, is a subject of the next subsection.
 
  Range restrictions on $q(z)$ implies estimates of (Taylor) coefficients of $q$ . This is the \textit{general principle }of Geometric Function Theory.

The best known problem of this kind is the famous \textit{Bieberbach conjecture} posed in 1916 by Ludwig Bieberbach \cite{bieber_c}---- if $g$ is an one-to-one analytic function, normalized by the condition $g'(0)=1$,  and $q(0)=0$, which  maps $\Delta$ onto the simple connected domain $\Omega\neq \mathbb{C}$  then $\vert \hat{g}(k) \vert \leq k$. This conjecture was proved only  in 1985 by Louis de Branges \cite{branges} and almost for the century was the main problem of Geometric Function Theory.

In this notes we do not need such sharp and deep results and will use two different but simple ways to obtain estimates of coefficients.

First way is rude and applied when $\Omega$ is a general (not simple-connected) domain. It is based on the  \textit{Vieta's formulas}.

Second way, when $\Omega$ is simple connected, is application of subordination principle for the Riemann mapping $g :   \Delta \to \Omega$. There is at least three elementary cases when we can obtain such estimates.

First one is general \textit{Lindelof  principle} for the first coefficient, second one is a little modification of \textit{Caratheodory lemma} and can be applied when $\Omega_M$ is a half-plane and the third one is  \textit{Rogosinskii inequality}. The last one we can apply for any simply connected domain, but this inequality gives only average domination of coefficients.

Finally (after we get desired estimates of Taylor coefficients) a normalized condition for $q\in \mathcal{Q}^T_n$  implies   estimate from below on the number $n$ 
.
\subsection{$M$ is a  point}
First lemma in this direction in the simplest form state that some  \textit{restrictions on the range} $q(\Delta)$ implies \textit{estimates of coefficients}. Namely, if a polynomial $q(z)$ satisfied $q(0)=0$ and omit some value in $\Delta$ then it coefficients  bounded by the the binomial coefficients times this value.

\begin{lemma}
\label{lemma_fc} Let $q(z)=\hat{q}(1)z +\hat{q}(2)z^2 +  \dots + \hat{q}(n) z^n$ and 
\begin{equation}
\label{l2}
w \notin q(\Delta)
\end{equation}
Then
\begin{equation}
\label{l3}
\vert \hat{q} (k) \vert \leq \binom {n} k \vert w \vert
\end{equation}

\end{lemma}
\begin{proof}
Let $\chi (z)=q(z)-w$ and $\chi^*(z)=\hat{q}(n) + \hat{q} (n-1)  z + \dots  + \hat{q} (1) z^{n-1} - w z ^n$  is it $n$-inverse. Then, $ 0 \notin \chi(\Delta)$ and, according to lemma \ref{lemma_1}, all roots of $\chi^*(z)=-w(z-\zeta_1) \dots (z-\zeta_n)$ lies in the $\bar{\Delta}$. Hence, by  Vieta's formulas,

$$
\vert \hat{q}(k) \vert =\vert    \sum_{1 \leq i_1<i_2< \dots i_k \leq n } (-1)^{k+1} w  \zeta_{i_1 }\zeta_{i_2} \dots \zeta_{i_k}\vert \leq \vert w\vert \sum_{1 \leq i_1<i_2< \dots i_k \leq n } 1 = \binom {n} k \vert w \vert
$$

\end{proof}
The estimates (\ref{l3}) are  best possible, since for the polynomial $q(z)=w-w(1-z)^n$ point $w \notin q(\Delta)$, but $\vert \hat{q}(k) \vert
=\binom {n} k \vert w \vert$.

Let $M=\mu$,  $\Omega=\hat{\mathbb{C}}\setminus \mu^{-1}$ and $q\in \mathcal{Q}_n$ is such that $q(\Delta)\subset \Omega$.   By the previous lemma 
$$
1=\sum_{k=1}^n \hat{q}(k)\leq \sum_{k=1}^n \vert\hat{q}(k)\vert\leq \sum_{k=1}^n\binom {n} k \vert \mu^{-1}\vert\leq \vert \mu \vert^{-1}(2^n-1)
$$
and hence
\begin{equation}
\label{coeff_ineq}
n\geq \log_2  (\vert \mu \vert +1)
\end{equation}

This is a first example of \textit{necessary conditions} which we write us
\begin{theorem} Let $p\in \mathcal{P}_n$ and $\mu \in S(p)$. Then $n\geq \log_2  (\vert \mu \vert +1)$.

\end{theorem}
 This fact, of course, can be deduced directly from Theorem \ref{theo_1}.

We conclude this subsection by the interesting observation concerning analogue of Koebe 1/4 Theorem for polynomials.

Because  $q : \Delta \to \mathbb{C} $ is an open mapping, the image $q(\Delta_r (z_0))$ always contains some disk $\Delta_\epsilon  (q(z_0))$. 
 But it seems interesting that lemma \ref{lemma_fc} implies this fact directly and from this we can conclude that $q(z)$ is an open mapping. It also gives the precise value of $\epsilon$, depending of $r$ and coefficients of the given polynomial. Namely, normalize a polynomial $q(z)$ by the condition $q(0)=0$ and define the number
$$
n(q)= \max_{1\leq k \leq n} \frac{ \vert \hat{q} (k)\vert }{\binom{n} k}
$$ 
Then lemma \ref{lemma_fc} implies that
\begin{equation}
\label{radsh1}
\Delta_{n(q)} \subseteq q(\Delta)
\end{equation}

Denote by $ \mathcal{U}_n=\lbrace q(z): q(z)=z+ \cdots + q_n z^n\rbrace $ the set of polynomials normalized by the conditions $q(0)=0$ and $q'(0)=1$. It is clear that  $n(q) \geq 1/n$ for $q \in \mathcal{U}_n $  and (\ref{radsh1}) implies
 \begin{equation}
\label{radsh3}
\Delta_{1/n} \subseteq q(\Delta)
\end{equation}

In \cite{dub} and \cite{dub2} was posed a problem on the maximal radius of the disk with the center at the origin, such that $q(\Delta)$ contained this disk for every $q\in \mathcal{U}_n$. Previous inclusion (\ref{radsh3})  gives the answer to this problem and example 
$$
q(z)=\frac{1}{n}(1-(1-z)^n)
$$
shows that the constant $1/n$ is the best possible.

Since  $n(q) \geq 1/(2^n-1)$ for $q \in \mathcal{Q}_n $ we have 
\begin{equation}
\label{radsh2}
\Delta_{\frac{1}{2^n-1}} \subseteq q(\Delta)
\end{equation}

As we mentioned above, all inclusion are best possible. Meanwhile, we shall to claim, that the set $\mathcal{U}_n$ is rotational invariant  and hence
$$
\bigcap_{q \in \mathcal{U}_n} q(\Delta)=\Delta_{1/n}
$$ 
but $\mathcal{Q}_n$ is not and we can improve a little the    inclusion (\ref{radsh2}).

 Let $\mathfrak{s}_n$ be a closed domain bounded by the curve
$$
c(\varphi)= \frac{1}{1-(2 \cos \varphi / n )^n e ^{i \varphi}} ~~~ \text{for} ~~~~ -\pi \leq \varphi \leq \pi
$$
 Then 

$$
\bigcap_{q \in \mathcal{Q}_n} q(\bar{\Delta})=\mathfrak{s}_n \cup \lbrace 1 \rbrace
$$

Indeed Theorem \ref{theo2} for $T=1$ and Theorem \ref{theo_1} implies that
$$
 \bigcap_{q \in \mathcal{Q}_n} q(\bar{\Delta}) =(\hat{\mathbb{C} }\setminus  \bigcup_{p \in \mathcal{P}_n} S(p))^*=(\hat{\mathbb{C}}\setminus S_n)^*=\mathfrak{s}_n \cup \lbrace 1 \rbrace
 $$

\begin{remark} We stress out here, that for $n \geq 3$ there is \textit{no one $q$ for which}
 $$
q(\bar{\Delta})  = \mathfrak{s}_n \cup \lbrace 1 \rbrace
 $$
because the set in the right side is not connected set of the complex plane and for big $n$ looks like a very small oval around $0$ and isolated point $\lbrace 1 \rbrace$.
\end{remark}

Let now $\Omega$ be a simply connected domain, $\Omega\neq \mathbb{C}$ and $0 \in \Omega$.  We shall to prove that condition 
\begin{equation}
\label{sub_1}
q(\Delta)\subset \Omega
\end{equation}
for $q \in \mathcal{Q}_n$  implies some bounds from below for the degree $n$ of the polynomial $q(z)$.

This bound  depend not only of the distance from $\Omega$ to the origin,which is obvious,  but also mainly of the geometric properties of $\Omega$.

This will be done in three steps.

Step 1. Restriction on  the range (\ref{sub_1})
with the condition $q(0)=0$ implies that $q(z)$ subordinate to the Riemann function $g_\Omega(z)$, which maps $\Delta$ onto $\Omega$ conformally and such that $g_\Omega(0)=0$. This is not a proposition but just \textit{a definition} (see Appendix).

Step 2. We compute the coefficients of $g_{\Omega}(z)$ and from subordination $q(z)\prec g_{\Omega}(z) $  conclude desired estimates for coefficients of $q(z)$. This is one part of principle of subordination or \textit{Lindelof principle} (\cite{littlewood} p. 171 and Appendix).

Step 3. Now normalization condition on $q$ expressed in terms of Taylor coefficients with estimates of coefficients implies the desired bound for $n$.

In what follows  $q \in \mathcal{Q}_n$ (i. e. $q(0)=0$, $q(1)=1$ and $\deg q \leq n$) 

\subsection{$\Omega_M$ is a half-plane}  We start with  $\Omega=\Pi_{\lambda}=\lbrace z : \Re z > -\lambda \rbrace$, where $\lambda > 0$. i.e. the right half-plane bounded by the vertical line at the distance $\lambda$ from the origin.

Let  $q(\Delta) \subset \Pi_\lambda$. Then the degree of the polynomial $q(z)$ could not be too small, namely 
\begin{equation}
\label{lambda_half_plane}
n\geq \frac{1}{2 \lambda}
\end{equation}

To prove this, define $ q_\lambda(z)=\frac{1}{2 \lambda} q(z)$. Then $q_\lambda(0)=0$ and 
$q_\lambda (\Delta) \subset \Pi_{1/2}$.  Hence, according to Caratheodori lemma (see Appendix) $\vert \hat{q}_\lambda(k) \vert \leq 1$. Thus  
$$
1=q(1)\leq \sum_{k=1}^n \vert \hat{q}(k) \vert =2 \lambda \sum_{k=1}^n \vert \hat{q}_\lambda(k) \vert \leq 2 n \lambda
$$
and (\ref{lambda_half_plane}) follows.

We recall, that if we prescribe that $q(\Delta)$  omit \textit{only one point } $-\lambda$ then  the degree of the polynomial $q(z)$ could be quite small, i.e. of the order $\log 1/ \lambda$. Now we see that if we prescribe that $q(\Delta)$ omit \textit{all left half-plane } $\Pi_\lambda$  then the degree of polynomial $q(z)$ have to be much bigger  --- \textit{at least} $1/2\lambda $.

\subsection{$\Omega_M$ is a slit domain}
Let now $\Omega$ be a slit domain $\Omega(\lambda )=\mathbb{C}\setminus \Gamma_{-1}(\lambda) $ where for any $\zeta\in \partial \Delta$
\begin{equation}
\Gamma_\zeta(\lambda) =\lbrace z: \arg z =\arg \zeta ~~~ \text{and} ~~~ \lambda \leq  \vert z \vert \rbrace 
\end{equation}

Since the Koebe function
$$
k(z)=\frac{4z}{(1-z)^2}=(\frac{1+z}{1-z})^2 -1
$$
is a composition of linear transform from $\Delta$ onto right half plane,  mapping $w \to w^2$ and shift  by 1 to the left, it maps a unit circle $\Delta$ onto $ \Omega(1)$   \textit{conformally} . 

Thus the function $k_\lambda (z)=\lambda k(z)$ maps $\Delta$ conformally onto $\Omega(\lambda$.)

It is clear that
$$
k_\lambda (z)= 4 \lambda \sum_{k=1}^\infty k z^k
$$

and hence by (\ref{rogoz} ) from Appendix,  for every $q$, such that $q(\Delta) \subset \Omega_\lambda$ we have
\begin{equation}
\label{14}
 \sum_{k=1}^n \vert \hat{q}(k) \vert ^2 \leq 16 \lambda^2 \sum_{k=1}^n \ k  ^2= 16 \lambda^2\frac{n(n+1)(2n+1)}{6 }< 16 \lambda^2 n^3
\end{equation}
But for $q_n \in \mathcal{Q}_n$ 
\begin{equation}
\label{15}
1\leq  (\sum_{k=1}^n  |\hat{q}(k) |)^2 \leq n \sum_{k=1}^n \vert \hat{q}(k) \vert ^2
\end{equation}
Hence  (\ref{14}) and (\ref{15}) implies that
\begin{equation}
\label{16}
n \geq \frac{1}{2 \sqrt{\lambda}}
\end{equation}

\subsection{$\Omega_M$ is a $T$-slit domain}

Let $T$ be a natural number and $\Omega_T(\lambda)$ be a $T$-slit domain, i.e.
$$
\Omega_T (\lambda)= \mathbb{C}  \setminus \bigcup _{k=0}^{T-1}\Gamma_{\omega^{1/2}_T\omega_T^k} (\lambda) 
$$
where $\omega_n=e^{i\frac{2\pi}{T}}$ is a $T$-root of unity. Thus $\Omega_T(\lambda)$ is a complex plane slits by $T$ radial rays with slopes $2\pi k /T+\pi/T$ for $k=0, 1, \dots, T-1$.

 Indeed  (\ref{incl_1}) implies that  $q(z) \prec \lambda k_T(z)$, where $k_T(z)$ is the $T$th root transform of the Koebe function (see Appendix D)  and we can apply Rogosinskii inequality  . Thus, using (\ref{sym5}) from Appendix D to estimate the coefficients of Koebe  $T$-function, we get 
 $$
 1\leq \vert q(1) \vert \leq \sum_{k=0}^{n-1}  \vert \hat{q}(kT+1) \vert \leq 
 n^{1/2}  \left( \sum_{k=0}^{n-1}  \vert \hat{q}(kT+1) \vert ^2 \right) ^{1/2}  \leq  \lambda n^{1/2}  \left( \sum_{k=0}^{n-1}  \vert \hat{k}_T(kT+1) \vert ^2 \right) ^{1/2} \leq 
 $$
 $$
C'_T \lambda n^{1/2}  \left( \sum_{k=1}^{n-1} k^{\frac{4}{T} -2} \right) ^{1/2} \leq C''_T \lambda n^{2/T}
 $$
 which implies  (\ref{incl_2})

Now let $q(z)=zp(z^T)$, where  $p \in \mathcal{P} _n$ . Then $q(z)$ is $T$-symmetric.  Suppose, that
 \begin{equation}
 \label{incl_1}
 q(\Delta) \subset \Omega_T(\lambda)
 \end{equation}
 We shall to prove that then
 
 \begin{equation}
 \label{incl_2}
 n\geq c_T \lambda^{-T/2}
 \end{equation}

Now, according to the  Theorem \ref{theo2}  we can rewrite (\ref{incl_2}) like a necessary condition
 
 \begin{theorem} Let  $M$ be a negative real segment
 $$
 M=\lbrace z : - \mu_M \leq \Re z \leq -1, \Im z =0 \rbrace
 $$
and $T$ be a natural number (length of the cycle). Let  the set $\langle a \rangle =a_1,a_2, \dots a_n$ successfully $ \langle a, T \rangle $-stabilize any cycle of the length $T$ with multiplier $\mu\in M$.

Then
 \begin{equation}
 \label{necessary_T_cycle}
 n\geq c_T \sqrt{\mu_M}
 \end{equation}
 
 \end{theorem}
 
 \subsection{$\Omega_M$ is a general simply connected domain }
 In the previous parts we deduce estimates from below of the degree of the polynomial $q(z)$, such that  $q(\Delta)\subset \Omega$ from the estimates of Taylor coefficients of the Riemann function $g_\Omega(z)$.
In this chapter we will obtain the same estimates much easy and for any simple connected domain $\Omega$.

The reason to do this is that we are going to obtain necessary conditions for every simply connected domain $\Omega_M$. Hence in order to obtain these estimates in the previous way we have  to use deep  estimates for coefficients $\widehat{ g_\Omega}(k)$. But we can obtain the same results (in order) by the much simpler way.

For the function $f$, defined in $\Delta$ it is convenient to denote $f_r(z)=f(rz)$ , $M(f)=\max_{\vert z \vert \leq 1} \vert f(z)\vert$ and  $M_r(f)=M(f_r)$.

\begin{lemma}
\label{theorem_max_mod} Let $\Omega$ be a simple connected domain, contained $0$ and $g_\Omega(z)$ is the corresponding Riemann function. Then for every polynomial $q\in \mathcal{Q}_n$  inclusion 
\begin{equation}
\label{theo4_1}
q(\bar{\Delta}) \subset \Omega
\end{equation}
implies that 
\begin{equation}
\label{theo4_2}
 M_{1-\frac{1}{2n}}( g_\Omega) \geq \frac{1}{2}
\end{equation}
\end {lemma}
\begin{proof}
Indeed, inclusion (\ref{theo4_1})  implies that $q(z) \prec g_\Omega(z)$ and hence $q_r(z) \prec (g_\Omega)_r(z)$ for any $0<r<1$.

But the modulus of polynomial of degree $n$ almost not changing in the annulus $\lbrace z : 1-\frac{1}{2n} <  \vert z \vert  < 1\rbrace$, i. e.

\begin{equation}
\label{th3_1}
M_{1-\frac{1}{2n}}(q) \geq \frac{1}{2}M(q)
\end{equation}

Indeed, by the mean value theorem, for every point $z\in \Delta$ exist point $w\in \Delta$, such that
$$
q(z)-q(((1-\frac{1}{2n})z)=q'(w) \frac{1}{2n}
$$
and application of Bernstein inequality yield
$$
M(q-q_{1-\frac{1}{2n}})\leq \frac{1}{2n} M(q')\leq \frac{1}{2} M(q)
$$
which imply (\ref{th3_1}).

Hence
$$
1= q(1)\leq M(q)\leq 2 M_{1-\frac{1}{2n}} (q) \leq  2 M_{1-\frac{1}{2n}} (g_\Omega).
$$
\end{proof}

This proposition, simple as it looks, is a powerful method to the necessary conditions for simply connected domains $\Omega$.

To demonstrate this we consider  again the half-plane and the slit domain.

Let $\Omega=\Pi _{\lambda} = \lbrace z : \Re z > -\lambda \rbrace$ for some $\lambda > 0$ and $q(\Delta) \subset \Omega$, for some polynomial $q\in \mathcal{Q}_n$.  Since 
$$
g_\Omega(z)= \frac{2\lambda z}{1-z}
$$ from (\ref{theo4_2}) we conclude that $M(q) \leq 2 M_{2n}(g_\Omega) \leq 8 n \lambda$. Thus for $q\in \mathcal{Q}_n$ 
$$
1=q(1) \leq M(q) \leq 8n\lambda
$$
and 
\begin{equation}
\label{nec_cond_est_hp}
n \geq \frac{1}{8\lambda}
\end{equation}

Let, now $\Omega=\Omega_T(\lambda)$  be the $T$-slit domain (see  Appendix for definition). Then 
\begin{equation}
\label{sym2}
k_{T, \lambda}(z)=\frac{2^\frac{2}{T}\lambda z}{(1-z^T)^\frac{2}{T}}
\end{equation}
is the correspondence Riemann function . Since 
$$
M_{2nT}(k_{T, \lambda})\leq \frac{2^\frac{2}{T} \lambda }{(1- (1-\frac{1}{2nT})^T)^\frac{2}{T}} \leq (4T)^\frac{2}{T}\lambda n^\frac{2}{T}
$$
for $q\in \mathcal{Q}_{nT-T+1}^T$ we have 
$$
1=q(1) \leq M_{2nT}(k_{T, \lambda}) \leq (4T)^\frac{2}{T}\lambda n^\frac{2}{T}
$$
and hence 
\begin{equation}
\label{nec_cond_est_T}
n \geq\frac{1}{4T\lambda^{T/2}} 
\end{equation}

 Next proposition shows that the case of $\Omega=\Omega_1(\lambda)$ is worth over all simply connected domains.
 
 \begin{theorem} Let $M$ be a  multiplier set, such that $\Omega=\hat{\mathbb{C}}\setminus M^*$ is a simply connected domain contained $0$.
Let  the set $\langle a \rangle =a_1,a_2, \dots a_n$ successfully $ \langle a \rangle $-stabilize any fixed point  with multiplier $\mu\in M$.

Then
 \begin{equation}
 \label{necessary_fix}
 n \geq \frac{\sqrt{\mu_M}}{3\sqrt{3} }
 \end{equation}
 
 \end{theorem}
 
 \begin{proof}
 First of all we claim, that since corresponding Riemann function $g_\Omega(z)$ is univalent in the unit disk $\Delta$,  $g_\Omega(0)=0$ and $w \notin g_\Omega (\Delta)$ then $\vert g_\Omega'(0)\vert \leq 4 \vert  w \vert$. This is simple consequence of the Koebe \textit{one-quarter theorem }(theorem 1.3 in  \cite{carleson}).
 
 Now, \textit{distortion theorem }(theorem 1.6 in \cite{carleson}) says, that 
\begin{equation}
\label{dist}
 M((g_\Omega)_r )\leq \frac{r \vert g_\Omega'(0)\vert }{(1-r)^2} 
\end{equation}
and Lemma \ref{theorem_max_mod} implies
$$
1\leq M(q) \leq 2 M_{2n}(g_\Omega)\leq 2 (2 n)^2 \vert g_\Omega'(0)\vert \leq 2 (2 n)^2  4 \lambda(\Omega)=32 n^2 \lambda(\Omega)
$$
 A little refinement of the previous proof gives the constant $27$ instead of $32$, but of course it is not the best possible one and this why we omit details.
 
 \end{proof}

\section{Simple solution for the $T$-slit domain $\Omega$} 

In this section for any $T$ and any $N$ we propose $\langle a , T\rangle$-stabilization  sequence, which is quite simple, straightforward and can be easy  calculate  in computer. Also it is very stable in the sense that it is not sensitive to the stabilisation sequence.

The construction based on the very simple ''seed''  -- some univalent polynomial $Q$, which maps $\Delta$ to $\Delta$. Then we iterate this polynomial with the help that every iteration is again univalent map.
We claim, that if the function $g(z)$ is defined at the set $E$ and $g(E) \subseteq  E$ then we can define iteration of $g$ by the rule $g^m(z)=g(g^{m-1}(z))$ for $m=2,3, \dots $, where $g^1(z)=g(z)$. Then if $g$ is  univalent in $\Delta$ and
$$
g(\Delta) \subseteq \Delta.
$$
then for every natural $m$ the function $g^m$ is also univalent in $\Delta$.

A domain $D$  is said to be \textit{starlike} with respect to a fixed point $w_0$ in $D$, if for
any point $w$ in $D$ the straight line segment $w_0w$ also lies in $D$. If $f$ is univalent in $\Delta$
and maps $\Delta$ onto a starlike domain with respect to $w_0 = 0$, we shall
call function $f$ \textit{starlike univalent}. 

\begin{lemma} Let  $n$ be a natural number, $c_n=\frac{n}{n+1} $ and
$$
Q (z)=c_n( z+\frac{z^n}{n})
$$
Then $Q$ is starlike univalent map from $\Delta$ \textit{ into } $\Delta$.
\end{lemma}
\begin{proof}
Straightforward calculations showed that
\begin{equation}
\label{u1}
\frac{\partial}{\partial \theta } \arg Q(e^{i\theta})= \Re \frac{e^{i\theta} Q'(e^{i\theta})}{Q(e^{i\theta})}
\end{equation}
and hence  ( we denote $\zeta = e^{i\theta}$)
$$
 \frac{\partial}{\partial \theta } \arg Q(e^{i\theta})= \Re n \frac{n+n\zeta ^{n-1} +{\bar{\zeta}}^{n-1} +1 }{\vert n + \zeta ^{n-1}\vert ^2 }= n(n+1)\frac{1+ \cos (n-1)\theta}{\vert n + \zeta ^{n-1}\vert ^2} \geq 0
$$
This implies that $Q$ maps the unit circle $\partial \Delta$ to the curve $C$ which is turning about $0$ in the positive direction -- anti-clockwise.

Now we observe that $C$ turning about $0$ only ones. This means that the winding number of $C$ about $0$ equal one. Indeed, by the Principle of Argument (\cite{duren}, p. 3) this number equals the number of zeros of $Q$ in the unit disk $\Delta$, which is obviously  equal to 1.

Hence $C$ is a Jordan curve around starlike (and hence simple connected) domain $D$. For a given $w \in D$ the curve $C=Q(\zeta)$ turns around $w$ ones as point $\zeta$ traverses the unit circle in positive direction and, hence, by the Principle of Argument in the unit disc $\Delta$  there is only one root of the equation $Q(z)=w$ and so $Q$ is univalent in $\Delta$.

\end{proof}

We shall to point out here that the curve $C= Q(\partial \Delta)$ tangent the unit circle $\partial \Delta$ at the $n-1$ roots of unity and has $n-1$ cusps inside $\Delta$, which corresponds to the critical points of $Q$.

\textbf{Remark.}   David Brannan proves that the only \textbf{starlike univalent} polynomials in $\mathcal{P}_N$ with all critical points on $\partial \Delta$ are $c(z \pm z^n/n )$ (\cite{bran} , Theorem 2.1).

For the given natural $T$ denote by $\omega = e^{\frac{2\pi i}{T}}$ the primitive root of unity of order $T$ and by $\omega^k$,  $k=0, 1, \dots, T-1$ the corresponding $T$ roots of unity.

Let we slit the unit disk $\Delta$ by the $T$ segments $\Upsilon_k=[\epsilon \omega ^{1/2} \omega ^k, \omega ^{1/2} \omega ^k]$ and consider 
$$
\Delta(\epsilon, T)= \Delta \setminus \bigcup_{k=0}^{T-1} \Upsilon_k .
$$
The next theorem shows that for a given $\epsilon$ there is a polynomial $Q$ of degree at most $\epsilon ^{-\gamma}$ for some positive $\gamma=\gamma(T)$ which maps conformally $ \Delta$ into $ \Delta (\epsilon, T)$ and $Q( \omega ^k ) = \omega ^k$ for all roots of unity $\omega^k$, $k=0,1, \dots, T-1$.
\begin{theorem}
Let $T$ be a given natural number greater than $2$ and $N \geq T+1$. 
Then exist  $p \in \mathcal{P}_N^\mathbb{R+}$ with \textbf{explicit formula}, such that for 
$$
q(z)=zp(z^T)
$$

\begin{equation}
\label{flth}
\max_{z \in S(T,q)} \vert q(z)\vert \leq 2 N^{-\gamma} 
\end{equation}
where 
$$
\gamma = \frac{1}{(T+1 )\log_2 (T+1)}.
$$
\end{theorem}
\begin{proof} Let 
$$
p_1(z)=\frac{T+1}{T+2}\left( 1+ \frac{z}{T+1} \right)
$$
and
\begin{equation}
\label{cycle2}
q_1(z)=z p_1(z^T)
\end{equation}

Then, according to the lemma 2,  $q_1(z)$ is univalent in $ \Delta$. It is clear also, that $ q_1(\Delta) \subset   \Delta$, $q_1(\omega^k)=\omega^k$ and 
\begin{equation}
\label{cycle 3}
q_1(r \omega ^{1/2} \omega^k)= r \omega ^{1/2} \omega^k p_1(-r^T)
\end{equation}
for every $k=0,1, \dots, T-1$.

Without loss of generality we can concentrate only for $k=0$, i.e. for the ray $\lbrace z : \arg z = \frac{\pi}{T}\rbrace $.  

The function $q_1$ maps a segment  
$$
\lbrace z : z= \omega ^{1/2} r ,  0  \leq r \leq 1\rbrace 
$$ 
one to one (because it is univalent) to the segment 
$$
\lbrace z : z= \omega ^{1/2} r ,  0 \leq r  \leq \frac{T}{T+2} \rbrace 
$$
and
\begin{equation}
\label{cycle 4}
\vert q_1(r \omega ^{1/2} ) \vert \leq \frac{T+1}{T+2} r.
\end{equation}

Let 
\begin{equation}
\label{fl0}
(T+1)^m \leq N < (T+1)^{m+1}
\end{equation} 
 for some $m=1,2, \dots $. Define
\begin{equation}
\label{fl1}
q(z)=q_1^m(z).
\end{equation}
Then straightforward calculations showed, that $q(z)= z p(z^T)$ for some polynomial $p$ of degree at most $(T+1)^m$ with non-negative coefficients and $p(1)=q_1^m(1)=q_1(1)=1$. 

Hence, according to (\ref{fl0}), $p \in \mathcal{P}_N^\mathbb{R+} $ and we have to prove only  (\ref{flth}) (for k=0).

The main (but simple) observation is that $\arg q(z)=\frac{\pi}{T}$ \textbf{only} for $z=r\omega^{1/2}$. 

Indeed, let for some $z \in \bar{\Delta}$ we have $\arg q(z)= \pi / T$ (which means that $q(z) \in [0, \omega^{1/2}]$). Then $q_1(q_1^{m-1}(z))\in [0, \omega^{1/2}]$. But $q_1(w)\in [0, \omega^{1/2}]$ \textbf{if and only if }$w \in [0, \omega^{1/2}] $. The \textbf{''if''} part follows from (\ref{cycle 3}) and \textbf{''only if'' } from lemma 2, because $q_1$ is \textbf{starlike} univalent.

 Hence $q_1^{m-1} (z) \in [0, \omega^{1/2}]$ and we can continue until we obtain $z \in [0, \omega^{1/2}]$ and we done.
 
 Thus, in order to prove (\ref{flth})  (for $k=0$) we can consider only points $z=r \omega^{1/2}$.
 
 Then  (\ref{cycle 3}) and  (\ref{cycle 4}) implies
 $$
 \vert q(r \omega^{1/2}) \vert \leq \vert q_1( q_1^{m-1}(\omega^{1/2}) )\vert \leq \left( \frac{T+1}{T+2} \right) \vert q_1^{m-1}(\omega^{1/2}) \vert \leq \dots \leq \left( \frac{T+1}{T+2} \right)^m
$$
But  (\ref{fl0}) implies that  
$$
m >\log_{T+1}N -1
$$ 
and hence
$$
\left( \frac{T+1}{T+2} \right)^m = \left( 1 + \frac{1}{T+1} \right)^{- (T+1)\frac{m}{T+1}}<\left( \frac{1}{2}\right) ^ \frac{m}{T+1} < \left( \frac{1}{2}\right) ^ \frac{\log_{T+1}N -1}{T+1} < 2N^{-\gamma}
$$
for $\gamma= \frac{1}{(T+1) \log_2 (T+1)}$.
\end{proof}
We omit the details how this implies stabilization properties of the corresponding set.

 \section{Some classical solutions.} 
 \label{sec_solut_al_suf}
 
\subsection{Suffridge polynomials}

Ted Suffridge in \cite{suffridge} (see also \cite{duren},  p.  268)  constructed an univalent polynomials
 \begin{equation}
 \label{suf1}
 P(z)=P(z;n,1)= \sum_{k=1}^{n} (1- \frac{k-1}{n})\frac{\sin \frac{k \pi}{n+1}}{\sin \frac{\pi}{n+1}} z^k
 \end{equation}
Elementary calculations (see \cite{suffridge} ) showed that 
\begin{equation}
\label{suff1}
P(e^{i\theta})=\frac{n+1}{2n(\cos\theta-\cos\frac{\pi}{n+1})}+ i \frac{\sin \theta (1+e^{i(n+1)\theta})}{2n(\cos\theta-\cos\frac{\pi}{n+1})}
\end{equation}
Hence 
\begin{equation}
\label{suff2}
P(1)=\frac{n+1}{2n(1-\cos\frac{\pi}{n+1})}
\end{equation}
and
\begin{equation}
\label{suff3}
P(- 1)=-\frac{n+1}{2n(1+\cos\frac{\pi}{n+1})}
\end{equation}
Since $P(z)$ is univalent polynomial with real coefficients (so with symmetrical image) this implies 
\begin{equation}
\label{suff3}
P(\Delta)\subset \Omega( {-\gamma})
\end{equation}
with
$$
\gamma=\frac{n+1}{2n(1+\cos\frac{\pi}{n+1})}
$$
Here $\Omega( {-\gamma})$ is a slit domain, i.e. $\mathbb{C}$ with the slit from $-\infty$ to $-\gamma$.

Define polynomial
$$
q(z)=\frac{2n(1-\cos\frac{\pi}{n+1})}{n+1} P(z)
$$
Then, (\ref{suff2}) implies that $q \in \mathcal{Q}_n$ and from( \ref{suff3}) we have
\begin{equation}
\label{suff4}
q(\Delta)\subset \Omega_{-\lambda}
\end{equation}
for 
$$
\lambda=\lambda(n) =\frac{1-\cos\frac{\pi}{n+1}}{1+\cos\frac{\pi}{n+1}}=\tan^2\frac{\pi}{2(n+1)}
$$
This fact was rediscovered in \cite{dm1}.

Obviously 
$$
n \sim  \frac{\pi}{2\sqrt{\lambda}}           ~~~~~~~     (n\to \infty)
$$
and hence the sequence
$$
a_k=\frac{2n(1-\cos\frac{\pi}{n+1})}{n+1} (1- \frac{k-1}{n})\frac{\sin \frac{k \pi}{n+1}}{\sin \frac{\pi}{n+1}} 
$$
is the $\langle a \rangle$-stabilization sequence for 
any fixed point with multiplier in the real segment  $M=\lbrace z : \Im z =0, -\mu_M \leq \Re z \leq -1 \rbrace$, where  $\mu_M>1$  and
\begin{equation}
\label{sq}
\text{ord} \langle a \rangle  \asymp \sqrt{\mu_M}
\end{equation}

\subsection{Alexander's polynomial}

Define polynomial
$$
q_\mathit{a}(z) =\frac{1}{l(n)} (z+\frac{z^2}{2}+\dots + \frac{z^n}{n})
$$
where $l(n)= 1+ 1/2 + \cdots +1/n$
 which we call  \textit{Alexander's polynomial} since it mapping properties like a complex polynomial was first established in \cite{alex}. There was also observed that $q_a(z)$ is an univalent polynomial. Obviously $q_a\in \mathcal{Q}_n$ and it is easy to see that $q_a(z)$ maps $\Delta$ into the rectangle with the height $\approx (\log n)^{-1}$
 
From this we can conclude that the sequence 
$$
a_k=\frac{1}{l(n) k} ~~~~  \text{for} ~~~~ k=1,2, \dots n
$$
with $n\leq c(\theta) \exp \mu_M$ is the $ \langle a \rangle$ stabilization sequence for any unstable fix point  with multiplier $\mu$ in the set $M$, where  $M=\lbrace z ; \vert z \vert \leq \mu_M, | \arg \mu | \geq \theta \rbrace $ is a (wide) sector of a (big) circle with radius $\mu_M$.  
 
We shall claim also that $q_a(z)$ has the smallest area of $q(\Delta)$ among all univalent polynomials from $q \in \mathcal{Q}_n$. Indeed, in this case
\begin{equation}
\label{area1}
area(q(\Delta) )\geq \frac{\pi }{\l(n)}
\end{equation}
and this inequality is sharp over univalent polynomials, since for $ q_\mathit{a}(z)$ we have equality in (\ref{area1}).

To prove this  one can observe that  (\cite{biber} p. 134)
\begin{equation}
\label{area2}
area (q(\Delta ) )= \int_0^1 \int_0^{2 \pi} \vert q'(r e ^{i\varphi }) \vert ^2 r d\varphi dr=
\end{equation}
$$
\int_0^1 \int_0^{2 \pi} q'(r e ^{i\varphi }) \bar{ q'(r e ^{i\varphi }) }r d\varphi dr=2 \pi \int_0^1( \vert q_1 \vert^2 r + 2^2 \vert q_2\vert^2 r^3 + 3^2 \vert q_3 \vert^2r^5 + \dots + n^2 \vert q_n \vert^2 r^{2n-1}) dr=
$$
$$
\pi ( \vert q_1 \vert^2  + 2 \vert q_2\vert^2 + 3 \vert q_3 \vert^2 + \dots + n \vert q_n\vert^2 ) 
$$
By the Cauchy�-Schwarz inequality and (\ref{area2}) 
\begin{equation}
\label{area3}
1=\vert \sum_{k=1}^{n} q_k\vert  \leq \sum_{k=1}^{n} \vert q_k \vert k^{1/2}k^{-{1/2}}\leq \left( \frac{area (q(\Delta ) )}{\pi}\right)^{1/2}( l(n))^{1/2}
\end{equation}
which imply (\ref{area1}).

\section{General solution for simply connected domain $\Omega$} 
In this section we shall use approximation theorem of V. Andrievskii and S. Ruschenweyh \cite{andru} to construct almost optimal polynomials for every simply connected domain $\Omega=\Omega_{M^*}$ .

We call this fact Theorem A and include it (without proof) in Appendix.

Let $\Omega$ be a simple connected domain in $\mathbb{C}$ contained points $0$ and $1$. Our aim is to find polynomial $q(z)$ of the possible smallest degree $n$ such that $q(0)=0$, $q(1)=1$ and 
\begin{equation}
\label{app_incl}
q(\bar{\Delta})\subseteq \Omega
\end{equation}

To find the desired polynomial we use Theorem A, which is  a marvellous  and power tool to construct univalent polynomials with prescribed set of values. The only disadvantage  is that this theorem do not give explicit formulas for coefficients.

Meanwhile, in our opinion, it is not too hard to improve this theorem to obtain the better constant and simply coefficients. This will be a subject of forthcoming paper.
 
\begin{theorem}
\label{theo_general}
Let $\Omega$ be a simply connected domain contained points $0$ and $1$ and $g_\Omega(z)$ be a corresponding  Riemann function normalized by the conditions $g_\Omega^{-1}(0)=0$ and $g_\Omega^{-1}(1)>0$.

Let $r=g_\Omega^{-1}(1)$ and $n$ be a smallest integer number satisfied 
\begin{equation}
\label{choose_number}
n \geq c_0 \max (\frac{1}{1-r},2)
\end{equation}
where $c_0$ is the constant from  Theorem A. 

Then there exist polynomial $q\in \mathcal{Q}_n$ such that
\begin{equation}
\label{inclusion}
q(\bar{\Delta})\subset \Omega
\end{equation}

Moreover, if a domain $\Omega$ is a $T$-symmetric then polynomial $q(z)$ has a real coefficients and l $q\in \mathcal{Q}_m^T$, where the number $m$ satisfied $mT-T+1=n$. In the case that $\Omega$ is symmetric with respect to the real line, then we also can chose $q(z)$ with real coefficients.
\end{theorem}
\begin{proof}

In view of Theorem A, and inequality (\ref{choose_number}) there exist polynomial $q_g(z)$ of degree less or equal $n$ with $q_g(0)=0$, and
\begin{equation}
\label{subord_double}
g_\Omega(rz)\prec q_g(z) \prec g_\Omega(z)
\end{equation}

 By definition of subordination $g_\Omega(rz)=q_g(\omega(z))$ for some analytic $\omega : \Delta \to \Delta$, such that $\omega(z) < \vert z \vert $ .
Let $\gamma=\omega(1)   $ and define polynomial $q(z)=q_g(\gamma z)$. Then, obviously $q\in \mathcal{Q}_n$ and since it defined by Theorem A like a linear means of Taylor series of $g_\Omega(z)$, it is $T$-symmetric if $\Omega$ is $T$-symmetric. Thus $q(z)\in\mathcal{Q}_n^T$.

Since $\vert \gamma \vert <1$ we have that $q(\bar{\Delta})=q_g(\gamma \bar{\Delta})\subset \Omega$
\end{proof}

\section{Appendix}

\subsection{Subordination, Lindelof principle, Rogozinski's and Caratheodory's lemmas}

We shall say that (analytic) function $f(z)$  \textit{subordinate} to (analytic) function $g(z)$ in $\Delta$ and denote this relation by  $f(z)\prec g(z)$ if $f(z) = g(\omega(z))$  for some analytic  $\omega : \Delta \to \Delta $, such that $\omega(0)=0$ (\cite{littlewood} p. 163). 

Note that for the univalent $g$ a function $f$  subordinate to $g$ if and only if $f(0)=g(0)$ and 
\begin{equation}
\label{app_a_sub}
f(\Delta) \subseteq  g(\Delta)
\end{equation}

Indeed if (\ref{app_a_sub}) hold,  define a function $\omega(z)=g^{-1} (f(z))$ in $\Delta$. Since it is regular in $\Delta$, maps $\Delta$ in $ \Delta$ and $\omega(0)=0$, we have that $f(z)$ subordinate to $g(z)$ in $\Delta$.

On the other hand, if $f(z)\prec g(z)$ then $f(z) = g(\omega (z))$  for some analytic  $\omega : \Delta \to \Delta $ and (\ref{app_a_sub}) follows from the Schwarz lemma.

 It is not true in general that subordination  $f(z) \prec g(z)$  implies estimate
\begin{equation}
\label{sub_2}
 \vert\hat{f} (n)\vert \leq\vert \hat{g} (n)\vert
\end{equation}
 for all $n$, as a simple example $f(z)=z^2$ and $g(z)=z$ shows. 
 
 Meanwhile (\ref{sub_2}) is always true for $n=1$: if $f$ subordinate $g$ than distortion of $f$ at $0$ is less then distortion of $g$. In other words,
 \begin{equation}
 \label{schwarz}
 \vert \hat{f}(1) \vert \leq \vert \hat{g}(1) \vert
 \end{equation}
 This is the \textit{Lindelof principle}.
 
 The proof is a simple application of \textit{Schwarz lemma}. 
Indeed, if $f$ subordinate $g$ then $f(z)=g(\omega(z))$ and 
 $$
 \hat{f}(1)= f'(0)=g'(0) \omega'(0)=\hat{g}(1) \omega'(0)
 $$
 But, according to the Schwarz lemma (\cite{duren}, p. 3) $\vert \omega(z) \vert \leq \vert z \vert $ and hence
 $$
 \vert \omega ' (0) \vert \leq 1
 $$
 
 It seems surprisingly that for \textit{any pair }of subordinate functions we still can give estimate  of coefficients at least in average. This is \textit{Rogozinski's inequality }  (\cite{littlewood}, p. 168 or \cite{duren}, p. 192) which we shall use in the sequel and stand without proof like
\begin{rogosinskii} 
If $f(z) \prec g(z)$ then
 \begin{equation}
 \label{rogoz}
 \sum_{k=1}^n \vert \hat{f}(k) \vert ^2 \leq \sum_{k=1}^n \vert \hat{g}(k) \vert ^2
 \end{equation}
\end{rogosinskii}
 
Let now function $f $ is subordinate to the \textit{univalent} function $g$ and the range $g (\Delta)$  \textit{ is convex}. Normalize $g$  by the conditions: $g(0)=0$ and $g'(0)=1$. Then all coefficients of $f$ are bounded by $1$. 
 
There is a simple proof of this remarkable fact, due to \textit{\textbf{Rogosinski}} (\cite{duren}, p. 195) .

Consider a transformation of $f$ which kills all coefficients $\hat{f}(k)$ for $k \notin n\mathbb{Z} $ and do not change coefficients for  $k \in n\mathbb{Z} $ :
\begin{equation}
\label{killer}
f_n(z)=\frac{1}{n}\sum_{k=0}^{n-1} f(\omega_n^k z)
\end{equation}
where $\omega_n=e^{i \frac{2\pi}{n}}$. Since $\Omega=g(\Delta)$ is convex, $f_n(z)=h(z^n)=\hat{f}(n)z^n+\hat{f}(2n)z^{2n} + \dots $ maps $\Delta$ into $\Omega$. Hence $h$ also maps $\Delta$ into $\Omega$ and it follows that $h \prec g$. Therefore, by Lindelof principle (\ref{schwarz}) $\vert \hat{f}(n) \vert = \vert\hat{h}(1)\vert \leq\vert \hat{g}(1)\vert =1$.

The extremal case is when  $\Omega$ is a half-plane $\Pi_{-1/2} = \lbrace z : \Re z > -1/2 \rbrace$ and $g$ is a conformal mapping of $\Delta$ onto $\Pi_{-1/2} $, such that $g (0)=0$, $g'(0)=1$. 

It  is easy to see that
 $$
g(z)=\frac{1}{2}(\frac{1+z}{1-z}-1)=\frac{z}{1-z}=z+z^2+z^3+ \dots
$$
and according to Schwarz lemma there is only one such mapping.

We see that
\begin{equation}
\label{g_by_one}
\hat{g}(k)=1~~~~ \mbox{for all natural}~~ k
\end{equation}
and in accordance with the general philosophy, that \textit{restriction on the image implies restriction on coefficients}, for \textit{any function which map} $\Delta$ \textit{into }$\Pi_{-\frac{1}{2}}$ \textit{and} $f(0)=0$ all Taylor coefficients is bounded by $1$.

This particular case (of general fact for convex domains) is called \textit{\textbf{Caratheodory's lemma} }(\cite{duren}, p. 41) which we shall use in the sequel and state as
\begin{caratheodory}
\label{lemma_car}
Let for analytic function $f$ 
$$
f(\Delta) \subseteq \Pi_{-\frac{1}{2}}
$$ 
and $f(0)=0$ . Then
\begin{equation}
\label{general_by_one}
\vert \hat{f}(k)\vert \leq 1~~~~ \mbox{for all natural}~~ k
\end{equation}
\end{caratheodory}

\subsection{ $T$-symmetric functions, $T$-th root transform, $T$-fold symmetry and Koebe function $k_T(z)$}

Function $g$ which is regular  in $\Delta$ with the  spectrum in the arithmetic progression $T\mathbb{Z} +1$, which means that $g(z)$ is of the form
\begin{equation}
\label{sym1}
g(z)= g_1z + g_{T+1}z^{T+1} + g_{2T+1} z^{2T+1} + g_{3T+1}z^{3T+1} +   \dots 
\end{equation}
is called \textit{ $T$-symmetric} ( see \cite{hayman}, p. 95).

This definition motivated by the fact that if function $g$ is $T$-symmetric  then the image $g(\Delta)$ has  $T$-fold rotational symmetry, that is invariant under the rotation $w\rightarrow \omega_T  w$, where $\omega_T=e^\frac{2\pi i}{T}$ is the $T$-root of unity. 

Indeed, if $w \in g(\Delta)$ then there is a point $ z\in \Delta$ such that $w=g(z)$ and according to (\ref{sym1})
$$
w= \omega_T^{-1} g(\omega_T z )
$$
Since $\omega_T z \in \Delta$ we have  $\omega_T w \in g(\Delta)$.

We stress out that converse  is not true. Indeed, the circle $\Delta$ obviously has a $T$-fold symmetry for any $T$ but $g(z)=z^n$ maps $\Delta$ onto $\Delta$ also  for any $n$.

Meanwhile, if the function $g(z)$ is\textit{ univalent} then converse is true.

We need this (well-known) fact later and include the proof.

\begin{lemma}
\label{lemma_T-symmetric_converse}
Let $g(z)$ be an univalent function, such that $g(0)=0$ and $g(\Delta) =\Omega$, where $\Omega$ has $T$-fold rotational symmetry. Then $g(z)$ is $T$-symmetric function.
\end{lemma}
\begin{proof}
Let
$$
g(z)=g_1z+g_2z^2+g_3z^3+ \dots
$$
Since $g(z)$ is univalent $g_1\neq 0$. 

Define 
$$
g_\theta(z)=g(e^{i\theta} z)= g'_1z+ g'_2z^2+g'_3z^3+ \dots 
$$
where $\theta$ is chosen in such a way, that $g'_1 >0$. We claim, that
\begin{equation}
\label{T-symm_mod}
\vert g'_k \vert =\vert g_k \vert
\end{equation}
Function $g_\theta(z)$ is univalent, maps $\Delta$ onto $\Omega$, $g_\theta(0)=0$ and $g'_\theta(0) > 0$. Thus $g_\theta (z)$ is the \textit{unique} Riemann mapping of $\Delta$ onto $\Omega$.

Define
$$
g^*(z)=\omega^{-1}_T g_\theta (\omega_T z)=
$$ 
$$
g'_1z+ \omega_Tg'_2z^2+\omega^2_Tg'_3z^3+ \dots  + \omega^{T-1}_Tg'_Tz^T+
$$
$$
g'_{T+1}z^{T+1}+ \omega_Tg'_{T+2}z^{T+2}+\omega^2_Tg'_{T+3}z^{T+3}+ \dots  + \omega^{T-1}_Tg'_{2T}z^{2T}+\dots 
$$
$$
g'_{mT+1}z^{mT+1}+ \omega_Tg'_{mT+2}z^{mT+2}+\omega^2_Tg'_{mT+3}z^{mT+3}+ \dots  + \omega^{T-1}_Tg'_{(m+1)T}z^{(m+1)T}+\dots 
$$
Since $\Omega$ has a $T$-fold rotational symmetry, $g^*(z)$ also maps $\Delta$ onto $\Omega$ conformally and because  $g^*(0)=g_\theta(0)=0$, $(g^*(0))'=g'_\theta(0)$ we have
$$
g^*(z)\equiv g_\theta(z)
$$
and hence for every $m=1,2,3,\dots$ and $j=2,3,\dots T$
\begin{equation}
\label{lemma_T-symmetric_coef}
\omega^{j}_Tg'_{mt+j}=g'_{mt+j}
\end{equation}
Since $\omega^{j}_T\neq 0$ for every $j=2,3,\dots T$  we conclude from (\ref{lemma_T-symmetric_coef}) that $g'_{mt+j}=0$ for every $m=1,2,3,\dots$ and $j=2,3,\dots T$. Hence (\ref{T-symm_mod}) implies that $g(z)$ is $T$-symmetric.
\end{proof}

Let now $g$ be an \textit{univalent} function in $\Delta$, such that $g(0)=0$. We define the \textit{$T$-th root transform} of $g$ by the formula
$$
g_{1/T} (z) =(g(z^T))^\frac{1}{T}
$$ 
(see \cite{duren} p. 28, where there is also the words of explanation, which we briefly repeat here). 

Since $g$ is univalent the image $g(\Delta)$ is a simple connected domain contained zero and
$$
g(z)=\hat{g}(1) z+ \hat{g}(2) z^2+ \hat{g}(3) z^3 + \cdots
$$
with $\hat{g}(1) \neq 0$.
Hence $g(z^T)$ covered $g(\Delta)$ exactly $T$ times and 
$$
g_{1/T} (z) =z(\hat{g}(1) +\hat{g}(2) z^T+ \hat{g}(3) z^{2T} + \cdots)^\frac{1}{T}=zh^\frac{1}{T}   (z)
$$
Since there is only one point in $\Delta$  where $g (z)$  is  equal to $0$ and this is $0$,  we can define the $T$-root of $h(z)$ near zero (where it is equal $\hat{g}(1) \neq 0$)  by the power series. It radius of convergence equal the distance to the closest singular point, which is the point where $h(z)=0$. But there is no other such points in $\Delta$, except zero. Hence the radius of convergence is not less  that $1$ and the root transform is well defined in $\Delta$. Since the power series contains only terms with $z^{kT+1}$, we conclude that $T$th-root transform  is a $T$-symmetric function.

It is clear also that $g_{1/T}(z)$ is univalent in $\Delta$.

Let now 
$$
r(z)=r_0 + r_1z + \dots + r_{n} z^{n}
$$
be a polynomial and
$$
q(z)=z(r(z))^T
$$

It is not hard to see that for the polynomial of this \textit{special kind} the $T$root transform defined correctly and 
$$
q_{1/T}(z)=zr(z^T)
$$

Let $\lambda>0$  and $\zeta\in \partial \Delta$. Define the ray 
\begin{equation}
\Gamma_\zeta(\lambda)=\lbrace z: \arg z =\arg \zeta ~~~ \text{and} ~~~ \lambda\leq  \vert z \vert \rbrace 
\end{equation}

For $T=1,2, \dots $ we  define the $T$-slit domains
$$
\Omega_T (\lambda)= \hat{\mathbb{C}}  \setminus \bigcup _{k=0}^{T-1}\Gamma_{\omega^{1/2}_T\omega_T^k} (\lambda) 
$$
Since the Koebe function
$$
k(z)=\frac{4z}{(1-z)^2}=(\frac{1+z}{1-z})^2 -1
$$
is a composition of linear transform from $\Delta$ onto $\Pi$,  mapping $w \to w^2$ and shift  by 1 to the left, it maps a unit circle $\Delta$ onto $ \Omega_1(1)$   \textit{conformally} . 

Let
\begin{equation}
\label{sym2}
k_T(z)=\frac{2^\frac{2}{T}z}{(1-z^T)^\frac{2}{T}}
\end{equation}
be a $T$th root transform of Koebe function. Then it is $T$-symmetric and hence maps $\Delta$ onto the slit domain $\Omega_T(1)$. Since Koebe function is univalent the same is $k_T(z)$ and hence this mapping is also conformal.

It is clear that 
\begin{equation}
\label{sym3}
k_T(z)= \sum_{k=0}^\infty c_k z^{kT+1}
\end{equation}
where
\begin{equation}
\label{sym4}
c_k=(-1)^k 2^\frac{2}{T} \binom {-\frac{2}{T}} k
\end{equation}
 Special cases are $T=1$, or the Koebe function, when $c_k=4k$ and  $T=2$, when $c_k=1$.

 There are few cases for so simple formulas, meanwhile for every $T$ we can  estimate coefficients as follows
 
 \begin{equation}
\label{sym5}
\vert c_k \vert \leq C_T  k^{\frac{2}{T} - 1}
\end{equation}
 
 Indeed,  this is a straightforward corollary  from the well known limit
 \begin{equation}
\label{sym6}
\lim_{n \to \infty} \frac{n! n^{z-1}}{(z)_n}=\Gamma (z)
\end{equation}
where $(a)_n=a(a+1)(a+2) \dots (a+n-1)$ is shifted factorial.

 \begin{remark}
In the view of the previous estimates it is natural to guess that for \textit{any} $T$-symmetric function, i.e. of the form (\ref{sym1}) which is \textbf{univalent}
$$
\mid \hat{q} (k) \mid \leq C_Tk^{\frac{2}{T}-1}
$$
This was a conjecture of G. Szego and was proved for  by J. Littlewood (1925) for $T=1$, by J. Littlewood and Paley (1932) for $T=2$, by V. Levin (1934) for $T=3$ and by A. Baernstein (1986) for $T=4$.  But J. Littlewood (1938) showed that it is false for sufficient large $T$ and later Ch. Pommerenke (1975) proved that conjecture is \textbf{false} for all $T\geq 12$. The cases $5\leq T \leq 11$ remain open (see \cite{hayman}, p.96)
\end{remark}
 
\subsection{Approximation Theorem of Andrievskij and Ruschenweyh}  Let $\Omega$ be a simple connected domain contained $0$ and $g_\Omega(z)$ be a corresponding \textit{Riemann function} which map $\Delta$ to $\Omega$ conformally and normalized by the conditions $g_\Omega(0)=0$ and $g_\Omega'(0)$ is positive and real.
Scwartz lemma implies that it is unique.

Next theorem, proved by V. Andrievskii and S. Ruschenweyh tell us how to approximate conformal mapping by polynomial map from inside.

\begin{theorem A}[ \cite{andru}] 
There exist a universal constant $c_0$, such that for every simply connected domain $\Omega$ contained point $0$ and corresponding Riemann function $g_\Omega(z)$ and for every $n\geq 2 c_0$ there exist a (univalent) polynomial of degree at most $n$ such that
\begin{equation}
g_\Omega((1-\frac{c_0}{n})z)\prec q(z) \prec g_\Omega(z)
\end{equation}
and moreover 
\begin{equation}
q(z)=\sum_{k=1}^n \gamma_kg_k z^k
\end{equation}
where
$$
g_\Omega(z)=\sum_{k=1} ^\infty g_kz^k
$$
and coefficients $\gamma_k$ are real.
\end{theorem A}

\begin{remark} As was pointed out in the survey \cite{andru1}, p. 40,  R. Grainer  proved that $\pi\leq c_0\leq 73$.
\end{remark}

\end{document}